\documentclass[12pt]{article}
\usepackage{fullpage}
\usepackage{amsmath,amssymb,amsfonts,amsthm}
\usepackage[all,arc,knot,poly]{xy}
\usepackage{enumerate}
\usepackage{mathrsfs}
\usepackage{mathtools}
\usepackage{amsxtra}
\usepackage{needspace}
\usepackage{graphicx}
\usepackage{setspace}
\usepackage{hyphenat}

\DeclareMathOperator{\Crit}{\operatorname{Crit}}
\DeclareMathOperator{\Proj}{\operatorname{Proj}}
\DeclareMathOperator{\Zer}{\operatorname{Zer}}
\DeclareMathOperator{\Pol}{\operatorname{Pol}}
\DeclareMathOperator{\Sym}{\operatorname{Sym}}
\DeclareMathOperator{\Res}{\operatorname{Res}}
\DeclareMathOperator{\Hom}{\operatorname{Hom}}

\theoremstyle{definition}
\newtheorem{theorem}{Theorem}[section]
\newtheorem{definition}[theorem]{Definition}
\newtheorem{lemma}[theorem]{Lemma}
\newtheorem{proposition}[theorem]{Proposition}

\newtheorem{assumption}[theorem]{Assumption}

\title{Voros symbols as cluster coordinates}

\author{
  Dylan G.L. Allegretti
}

\date{}

\begin{document}

\maketitle

\begin{abstract}
We show that the Borel sums of the Voros symbols considered in the theory of exact WKB analysis arise naturally as Fock-Goncharov coordinates of framed $PGL_2(\mathbb{C})$-local systems on a marked bordered surface. Using this result, we show that these Borel sums can be meromorphically continued to any point of~$\mathbb{C}^*$, and we prove an asymptotic property of the monodromy map introduced in collaboration with Tom~Bridgeland.
\end{abstract}

\tableofcontents

\section{Introduction}

\subsection{Voros symbols and their cluster nature}

The WKB method was originally introduced by Wentzel, Kramers, and Brillouin in 1926 as a way of finding approximate solutions of the Schr\"odinger equation
\begin{align}
\label{eqn:schrodingerintro}
\hbar^2\frac{d^2}{dz^2}y(z,\hbar)-Q(z,\hbar)y(z,\hbar)=0
\end{align}
in the semiclassical limit $\hbar\ll1$. In 1983, Voros~\cite{Voros} described a reformulation of this method which can be used to construct exact solutions. In this reformulation, one begins by constructing formal series solutions of~\eqref{eqn:schrodingerintro}. These formal series are generally divergent, and genuine analytic solutions are obtained by taking Borel sums.

In this paper, we use the exact WKB method to study the Schr\"odinger equation defined on an arbitrary compact Riemann surface~$S$. By this we mean a collection of equations of the form~\eqref{eqn:schrodingerintro}, one for each local coordinate $z$ on~$S$. The potential function $Q(z,\hbar)$ is assumed to be of the form 
\[
Q(z,\hbar)=Q_0(z)+\hbar Q_1(z)+\hbar^2 Q_2(z)+\dots
\]
where each $Q_n(z)$ is a meromorphic function of~$z$ and $Q_n\equiv0$ for $n\gg0$. These local expressions are related in such a way that there is an associated meromorphic quadratic differential given in any local coordinate $z$ by the expression 
\[
\phi(z)=Q_0(z)dz^{\otimes2}.
\]
The properties of this quadratic differential are deeply related to the properties of the corresponding differential equation.

A meromorphic quadratic differential $\phi$ as above gives rise to an object called a \emph{marked bordered surface}. This is defined as a pair $(\mathbb{S},\mathbb{M})$ where $\mathbb{S}$ is a compact oriented surface with boundary and $\mathbb{M}$ is a collection of finitely many marked points on~$\mathbb{S}$ such that every boundary component contains at least one marked point. Given the quadratic differential $\phi$, we define a marked bordered surface where $\mathbb{S}$ is obtained by taking an oriented real blowup of~$S$ at each pole of order $\geq3$ and $\mathbb{M}$ consists of poles of order $\leq2$ together with points on~$\partial\mathbb{S}$ corresponding to the Stokes lines of the differential equation.

For a generic quadratic differential $\phi$, we also get an \emph{ideal triangulation} of the associated marked bordered surface. By this we mean a triangulation of $\mathbb{S}$ whose vertices are precisely the points of~$\mathbb{M}$. Given the differential $\phi$, we define an ideal triangulation by considering the foliation of $S$ whose leaves are the curves 
\[
\Im\int^z\sqrt{Q_0(z)}dz=\text{constant}.
\]
We are particularly interested in those leaves whose endpoints are zeros of the quadratic differential. A leaf which connects two zeros of~$\phi$ is known as a \emph{saddle trajectory}. By taking the differential $\phi$ to be generic, we may assume that there are no saddle trajectories in the foliation. Such a differential is said to be \emph{saddle-free}. For a complete saddle-free differential, any leaf of the foliation that has a zero as one of its endpoints has a pole as the other endpoint. The complement of these curves in~$S$ is a union of cells, examples of which are illustrated on the left below.
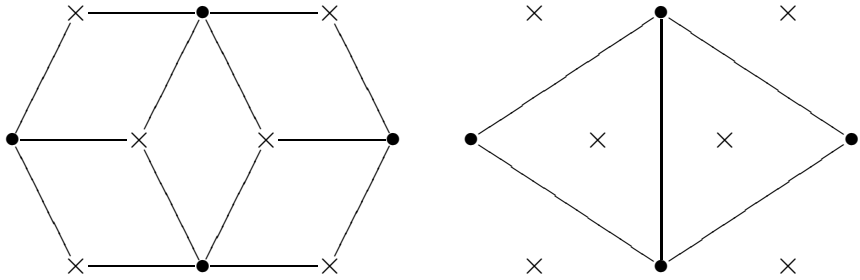
\[
\xy /l2pc/:
(2,-2)*{\times}="11";
(0,-2)*{\bullet}="21";
(-2,-2)*{\times}="31";
(3,0)*{\bullet}="12";
(1,0)*{\times}="22";
(-1,0)*{\times}="32";
(-3,0)*{\bullet}="42";
(2,2)*{\times}="13";
(0,2)*{\bullet}="23";
(-2,2)*{\times}="33";
{"11"\PATH~={**@{-}}'"21"},
{"21"\PATH~={**@{-}}'"31"},
{"12"\PATH~={**@{-}}'"22"},
{"32"\PATH~={**@{-}}'"42"},
{"13"\PATH~={**@{-}}'"23"},
{"23"\PATH~={**@{-}}'"33"},
{"11"\PATH~={**@{-}}'"12"},
{"21"\PATH~={**@{-}}'"22"},
{"21"\PATH~={**@{-}}'"32"},
{"31"\PATH~={**@{-}}'"42"},
{"12"\PATH~={**@{-}}'"13"},
{"22"\PATH~={**@{-}}'"23"},
{"32"\PATH~={**@{-}}'"23"},
{"42"\PATH~={**@{-}}'"33"},
\endxy
\qquad
\xy /l2pc/:
(2,-2)*{\times}="11";
(0,-2)*{\bullet}="21";
(-2,-2)*{\times}="31";
(3,0)*{\bullet}="12";
(1,0)*{\times}="22";
(-1,0)*{\times}="32";
(-3,0)*{\bullet}="42";
(2,2)*{\times}="13";
(0,2)*{\bullet}="23";
(-2,2)*{\times}="33";
{"12"\PATH~={**@{-}}'"21"},
{"21"\PATH~={**@{-}}'"42"},
{"12"\PATH~={**@{-}}'"23"},
{"23"\PATH~={**@{-}}'"42"},
{"21"\PATH~={**@{-}}'"23"},
\endxy
\]
Here we indicate the zeros of the quadratic differential by~$\times$ and the poles by~$\bullet$. Choosing a single leaf of the foliation within each cell as illustrated in the picture on the right, we obtain a triangulation of the associated marked bordered surface, well defined up to isotopy. This triangulation is called the \emph{WKB triangulation} of~$\phi$. We employ an extension of this concept called the \emph{signed WKB triangulation}.

Given the quadratic differential~$\phi$, we can consider the branched double cover $\pi:\Sigma_\phi\rightarrow S$ on which the 1-form $\sqrt{Q_0(z)}dz$ is single-valued. It is known as the \emph{spectral cover} for~$\phi$. Let $\Crit(\phi)$ denote the set of critical points of~$\phi$. Then the \emph{Voros symbol} associated to a class $\gamma\in H_1(\Sigma_\phi\setminus\pi^{-1}\Crit(\phi))$ is a formal series in~$\hbar$. It is defined as $e^{V_\gamma}$ where $V_\gamma$ is the period of a certain meromorphic 1-form integrated around the cycle~$\gamma$. Voros symbols play an important role in WKB analysis~\cite{KawaiTakei05}, where they appear in the explicit calculation of the monodromy group of equation~\eqref{eqn:schrodingerintro}. If $\phi$ is a saddle-free differential, then there is a canonical cycle~$\gamma_j$ associated to each arc $j$ of the WKB triangulation, and hence we have the associated Voros symbol~$e^{V_{\gamma_j}}$.

In~\cite{IwakiNakanishi}, Iwaki and Nakanishi studied how the Voros symbols change as the quadratic differential~$\phi$ varies. Their results can be formulated in terms of the one-parameter family of quadratic differentials $\phi^{(\theta)}$, $\theta\in[-r,r]$, described in Section~3.6 of~\cite{IwakiNakanishi}. The differential $\phi^{(\theta)}$ is saddle-free for $\theta\neq0$, but the associated foliation of~$S$ develops a saddle trajectory when $\theta=0$. The diagrams below illustrate the trajectory structure of the foliations associated to these differentials.
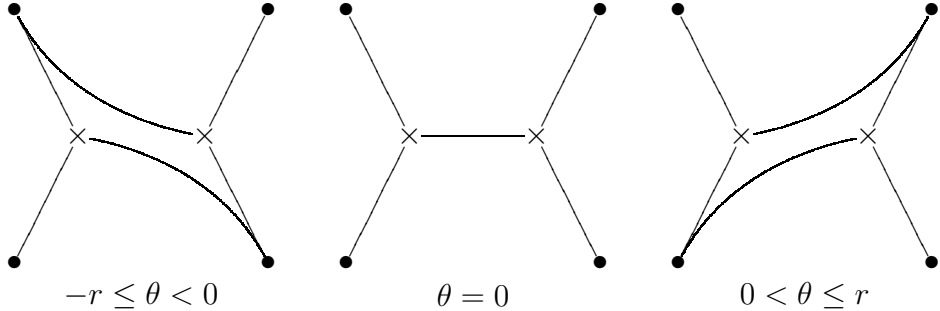
\[
\xy /l2pc/:
(2,-2)*{\bullet}="A";
(-2,-2)*{\bullet}="B";
(1,0)*{\times}="X";
(-1,0)*{\times}="Y";
(2,2)*{\bullet}="C";
(-2,2)*{\bullet}="D";
{"A"\PATH~={**@{-}}'"X"},
{"B"\PATH~={**@{-}}'"Y"},
{"C"\PATH~={**@{-}}'"X"},
{"D"\PATH~={**@{-}}'"Y"},
"A";"Y" **\crv{(1.5,-1) & (0.5,-0.25)};
"D";"X" **\crv{(-1.5,1) & (-0.5,0.25)};
(0,2.5)*{-r\leq\theta<0};
\endxy
\qquad
\xy /l2pc/:
(2,-2)*{\bullet}="A";
(-2,-2)*{\bullet}="B";
(1,0)*{\times}="X";
(-1,0)*{\times}="Y";
(2,2)*{\bullet}="C";
(-2,2)*{\bullet}="D";
{"X"\PATH~={**@{-}}'"Y"},
{"A"\PATH~={**@{-}}'"X"},
{"B"\PATH~={**@{-}}'"Y"},
{"C"\PATH~={**@{-}}'"X"},
{"D"\PATH~={**@{-}}'"Y"},
(0,2.5)*{\theta=0};
\endxy
\qquad
\xy /l2pc/:
(2,-2)*{\bullet}="A";
(-2,-2)*{\bullet}="B";
(1,0)*{\times}="X";
(-1,0)*{\times}="Y";
(2,2)*{\bullet}="C";
(-2,2)*{\bullet}="D";
{"A"\PATH~={**@{-}}'"X"},
{"B"\PATH~={**@{-}}'"Y"},
{"C"\PATH~={**@{-}}'"X"},
{"D"\PATH~={**@{-}}'"Y"},
"C";"Y" **\crv{(1.5,1) & (0.5,0.25)};
"B";"X" **\crv{(-1.5,-1) & (-0.5,-0.25)};
(0,2.5)*{0<\theta\leq r};
\endxy
\]
Note that the transformation of $\phi^{(-r)}$ to $\phi^{(r)}$ corresponds to a \emph{flip} of an arc of the WKB triangulation. This is the operation that removes an edge~$k$ and replaces it by the unique different edge that forms a new ideal triangulation:
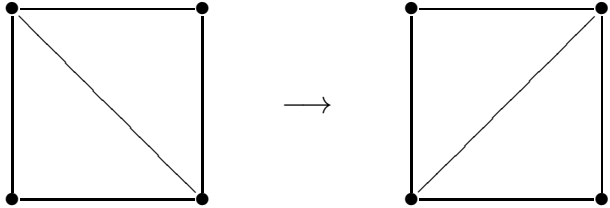
\[
\xy /l1.5pc/:
(2,-2)*{\bullet}="A";
(-2,-2)*{\bullet}="B";
(2,2)*{\bullet}="C";
(-2,2)*{\bullet}="D";
{"A"\PATH~={**@{-}}'"B"},
{"A"\PATH~={**@{-}}'"C"},
{"C"\PATH~={**@{-}}'"D"},
{"B"\PATH~={**@{-}}'"D"},
{"A"\PATH~={**@{-}}'"D"},
\endxy
\qquad
\longrightarrow
\qquad
\xy /l1.5pc/:
(2,-2)*{\bullet}="A";
(-2,-2)*{\bullet}="B";
(2,2)*{\bullet}="C";
(-2,2)*{\bullet}="D";
{"A"\PATH~={**@{-}}'"B"},
{"A"\PATH~={**@{-}}'"C"},
{"C"\PATH~={**@{-}}'"D"},
{"B"\PATH~={**@{-}}'"D"},
{"B"\PATH~={**@{-}}'"C"},
\endxy
\]
The spectral cover $\Sigma_{\phi^{(\theta)}}$ defined by a differential $\phi^{(\theta)}$ in this one-parameter family is independent of~$\theta$, and therefore so is the homology $H_1(\Sigma_{\phi^{(\theta)}}\setminus\pi^{-1}\Crit(\phi^{(\theta)}))$. Let us write $e^{V_\gamma^{(\theta)}}$ for the Voros symbol associated to the quadratic differential $\phi^{(\theta)}$ and class~$\gamma$. Building on results of~\cite{DDP93} and~\cite{KoikeSchafke}, Iwaki and Nakanishi proved the following.

\begin{theorem}[\cite{IwakiNakanishi}, Proposition~7.3]
\label{thm:transformationVoros}
Let $\phi^{(\theta)}$, $\theta\in[-r,r]$, be the family of quadratic differentials from~\cite{IwakiNakanishi}, Section~3.6. Then for every cycle $\gamma$, there is an equality of analytic functions 
\[
\lim_{\delta\rightarrow0^+}\mathcal{S}\left[e^{V_\gamma^{(-\delta)}}\right] = \lim_{\delta\rightarrow0^+}\mathcal{S}\left[e^{V_\gamma^{(\delta)}} \left(1+e^{V_{\gamma_k}^{(\delta)}}\right)^{-(\gamma_k,\gamma)}\right]
\]
where $(\cdot,\cdot)$ is the intersection pairing on homology and we write $\mathcal{S}[f]$ for the Borel sum of a formal series $f$.
\end{theorem}

Interestingly, the rational expression appearing on the right hand side of the equation in this theorem also appears in the theory of cluster algebras. In~\cite{IwakiNakanishi}, Iwaki and Nakanishi sharpened this observation by showing that the Voros symbols associated to cycles correspond to elements in the cluster algebra associated to a marked bordered surface by~\cite{FST}. Still, the relationship between Voros symbols and cluster algebras begs for some deeper explanation in terms of previously known cluster structures in mathematics.

\subsection{Cluster coordinates on moduli spaces of local systems}

In their seminal work in higher Teichm\"uller theory~\cite{FG1}, Fock and Goncharov defined versions of the moduli space of local systems on a surface and showed that their spaces admit atlases of rational cluster coordinates, making them birational to cluster varieties.

To be more precise, suppose that $(\mathbb{S},\mathbb{M})$ is a marked bordered surface, and let $\mathbb{S}'=\mathbb{S}\setminus\mathbb{M}$. Given a $PGL_2(\mathbb{C})$-local system $\mathcal{L}$ on~$\mathbb{S}'$, we can form the associated bundle 
\[
\mathcal{L}_{\mathbb{CP}^1}\coloneqq\mathcal{L}\times_{PGL_2(\mathbb{C})}\mathbb{CP}^1.
\]
Following~\cite{FG1}, we define a \emph{framed $PGL_2(\mathbb{C})$-local system} on~$(\mathbb{S},\mathbb{M})$ to consist of a $PGL_2(\mathbb{C})$-local system on~$\mathbb{S}'$ and, for each marked point $p\in\mathbb{M}$, a flat section of the restriction of $\mathcal{L}_{\mathbb{CP}^1}$ to a small neighborhood of~$p$. We denote by $\mathcal{X}(\mathbb{S},\mathbb{M})$ the moduli stack parametrizing all framed $PGL_2(\mathbb{C})$-local systems on~$(\mathbb{S},\mathbb{M})$. To define coordinates on this space, suppose we are given an ideal triangulation $T$ of~$(\mathbb{S},\mathbb{M})$. Let $j$ be an arc of this triangulation separating two ideal triangles as illustrated below. (As we will explain later, this construction can be generalized to the case where $j$ is an arc in a ``self-folded triangle''.)
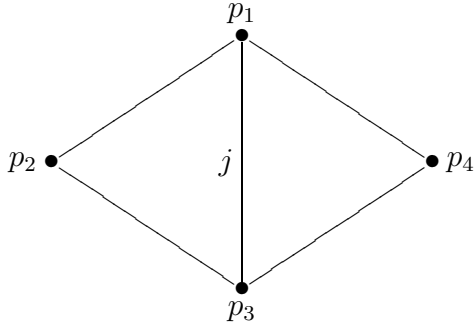
\[
\xy /l2pc/:
(0,-2)*{\bullet}="21";
(3,0)*{\bullet}="12";
(-3,0)*{\bullet}="42";
(0,2)*{\bullet}="23";
{"12"\PATH~={**@{-}}'"21"},
{"21"\PATH~={**@{-}}'"42"},
{"12"\PATH~={**@{-}}'"23"},
{"23"\PATH~={**@{-}}'"42"},
{"21"\PATH~={**@{-}}'"23"},
(0.25,0)*{j};
(0,-2.35)*{p_1};
(3.45,0)*{p_2};
(0,2.35)*{p_3};
(-3.45,0)*{p_4};
\endxy
\]
The choice of a point in $\mathcal{X}(\mathbb{S},\mathbb{M})$ provides a flat section of $\mathcal{L}_{\mathbb{CP}^1}$ near each marked point~$p_i$ in the picture above. If we parallel transport these sections to a common point in the interior of the quadrilateral, we get points $z_i\in\mathbb{CP}^1$ for $i=1,\dots,4$ which we view as lines in~$\mathbb{C}^2$. Fixing a volume form $\omega$ on~$\mathbb{C}^2$ and choosing a vector $v_i$ in each line $z_i$, we can form the cross ratio 
\[
X_j=\frac{\omega(v_1\wedge v_2)\omega(v_3\wedge v_4)}{\omega(v_2\wedge v_3)\omega(v_1\wedge v_4)}.
\]
The cross ratios $X_j$ for $j$ an arc of~$T$ are called the \emph{Fock-Goncharov coordinates} with respect to the triangulation. They provide a birational map 
\[
\mathcal{X}(\mathbb{S},\mathbb{M})\dashrightarrow(\mathbb{C}^*)^n
\]
where $n$ is the number of arcs in an ideal triangulation. One can show that this map depends only on the choice of ideal triangulation and not on any other choices in the construction.

Suppose $T$ and $T'$ are two ideal triangulations of~$(\mathbb{S},\mathbb{M})$ so that $T'$ is obtained from~$T$ by a flip of an edge~$k$. Then the coordinates associated to the triangulation $T'$ are obtained from the ones associated to~$T$ by a type of rational transformation known as a~\emph{cluster transformation}. This transformation is the composition of a simple monomial transformation with the rational map given on coordinates by 
\[
X_j\mapsto X_j(1+X_k)^{-\varepsilon_{jk}}
\]
where $\varepsilon_{st}$ is a skew-symmetric matrix indexed by edges of the ideal triangulation~$T$.

The rational expression appearing in the above map is exactly the one that appears on the right hand side of the formula in Theorem~\ref{thm:transformationVoros}. Thus it is natural to wonder if the Voros symbols can be realized as Fock-Goncharov coordinates of framed local systems. The purpose of this paper is to give such a realization. To do this, we employ the map recently introduced in~\cite{AllegrettiBridgeland1}, which takes meromorphic projective structures to their monodromy data.

\subsection{The monodromy map and main theorem}

Recall that a \emph{projective structure} on a Riemann surface $S$ is defined as an atlas of holomorphic charts mapping open sets in~$S$ into~$\mathbb{CP}^1$ such that the transition functions are restrictions of M\"obius transformations. The space of projective structures on a Riemann surface $S$ has the structure of an affine space modeled on the vector space $H^0(S,\omega_S^{\otimes2})$ of holomorphic quadratic differentials on~$S$. Indeed, if $\mathcal{P}_0$ is a projective structure with a local chart $z:U\rightarrow\mathbb{CP}^1$ and $\phi$ is a holomorphic quadratic differential, given in the local chart by the expression 
\[
\phi(z)=\varphi(z)dz^{\otimes2},
\]
then we get a chart $w:U\rightarrow\mathbb{CP}^1$ for a new projective structure $\mathcal{P}_0+\phi$ by taking linearly independent solutions $y_1(z)$ and $y_2(z)$ of the equation 
\begin{align}
\label{eqn:schrodingerwithoutparameterintro}
y''(z)-\varphi(z)y(z)=0
\end{align}
and setting $w\coloneqq y_1(z)/y_2(z)$.

In~\cite{AllegrettiBridgeland1}, we considered a moduli space parametrizing projective structures with poles of prescribed orders. To define such a projective structure, let us fix an ordinary projective structure $\mathcal{P}_0$ on~$S$. Then a \emph{meromorphic projective structure} is defined as a collection of charts $w:U\rightarrow\mathbb{CP}^1$ given by ratios of solutions of the equation~\eqref{eqn:schrodingerwithoutparameterintro} as above, where now the quadratic differential $\phi$ is allowed to have poles. The resulting object is again written $\mathcal{P}_0+\phi$. By a \emph{pole} of a meromorphic projective structure, we mean a pole of the differential~$\phi$. One can easily show that this notion is independent of the choice of~$\mathcal{P}_0$.

A meromorphic projective structure $\mathcal{P}$ on a Riemann surface~$S$ naturally determines a marked bordered surface. Writing $\mathcal{P}=\mathcal{P}_0+\phi$ where $\mathcal{P}_0$ is an ordinary projective structure and $\phi$ is a meromorphic quadratic differential, this is simply defined as the marked bordered surface associated to~$\phi$. One can show that this definition is independent of the choice of~$\mathcal{P}_0$. If $S$ is a Riemann surface and $\mathcal{P}$ is a meromorphic projective structure on~$S$, then a \emph{marking} of $(S,\mathcal{P})$ by~$(\mathbb{S},\mathbb{M})$ is an isomorphism of $(\mathbb{S},\mathbb{M})$ with the marked bordered surface induced by $\mathcal{P}$, considered up to diffeomorphisms isotopic to the identity. There is a moduli space $\Proj(\mathbb{S},\mathbb{M})$ parametrizing pairs $(S,\mathcal{P})$ as above together with a marking by~$(\mathbb{S},\mathbb{M})$.

It is convenient to modify this space in two ways, which we will explain in detail later. Firstly, we consider a dense open set 
\[
\Proj^\circ(\mathbb{S},\mathbb{M})\subseteq\Proj(\mathbb{S},\mathbb{M})
\]
which is the complement of the codimension two locus of projective structures with apparent singularities. Secondly, we consider a finite cover 
\[
\Proj^*(\mathbb{S},\mathbb{M})\rightarrow\Proj^\circ(\mathbb{S},\mathbb{M})
\]
of degree $2^N$ where $N$ is the number of marked points in the interior of~$\mathbb{S}$. In terms of the space $\Proj^*(\mathbb{S},\mathbb{M})$, the main result of~\cite{AllegrettiBridgeland1} was the following.

\begin{theorem}[\cite{AllegrettiBridgeland1}, Theorem~1.1]
Let $(\mathbb{S},\mathbb{M})$ be a marked bordered surface, and if $\mathbb{S}$ has genus zero, assume that $|\mathbb{M}|\geq3$. Then there exists a natural holomorphic map 
\[
F:\Proj^*(\mathbb{S},\mathbb{M})\rightarrow\mathcal{X}(\mathbb{S},\mathbb{M})
\]
sending a projective structure to its monodromy representation equipped with a natural framing given by Stokes data for the differential equation~\eqref{eqn:schrodingerwithoutparameterintro}.
\end{theorem}

For any $\varepsilon>0$, we write 
\[
\mathbb{H}(\varepsilon)=\{\hbar\in\mathbb{C}:|\hbar|<\varepsilon \text{ and } \Re(\hbar)>0\}.
\]
Given a meromorphic quadratic differential $\phi$, we define in Section~\ref{sec:VorosSymbolsAsClusterCoordinates} an associated meromorphic projective structure~$\mathcal{P}$, and we consider the sum 
\[
\mathcal{P}(\hbar)=\mathcal{P}+\frac{1}{\hbar^2}\phi
\]
for any $\hbar\in\mathbb{H}(\varepsilon)$. The charts of $\mathcal{P}(\hbar)$ are given by ratios of solutions of Schr\"odinger's equation. If we equip $\phi$ with the additional data of a signing and a marking as defined in Section~\ref{sec:VorosSymbolsAsClusterCoordinates}, then we can think of this as a point $\mathcal{P}(\hbar)\in\Proj^*(\mathbb{S},\mathbb{M})$. The assignment $\hbar\mapsto\mathcal{P}(\hbar)$ extends to a multivalued function on~$\mathbb{C}^*$ branched only at the origin. 

\begin{theorem}
\label{thm:intromain}
Let $\phi$ be a complete saddle-free GMN differential equipped with a signing and a marking by $(\mathbb{S},\mathbb{M})$. Then there exists $\varepsilon>0$ such that 
\begin{enumerate}
\item For all points $\hbar\in\mathbb{H}(\varepsilon)$, the Fock-Goncharov coordinates of $F(\mathcal{P}(\hbar))$ with respect to the signed WKB triangulation of~$\phi$ are well defined.
\item Taking the Fock-Goncharov coordinate associated to an arc of the signed WKB triangulation gives a holomorphic map $\mathbb{H}(\varepsilon)\rightarrow\mathbb{C}^*$ which agrees with the Borel sum of the corresponding Voros symbol.
\end{enumerate}
\end{theorem}

The first part of this theorem was proved in~\cite{Allegretti17} in the special case where $(\mathbb{S},\mathbb{M})$ is a disk with finitely many marked points on its boundary. It was used there to relate the wall-and-chamber decomposition of certain spaces of Bridgeland stability conditions to the algebraic tori glued together in the construction of cluster varieties.

\subsection{Meromorphic continuation and asymptotic property}

The Borel sum of any Voros symbol is holomorphic in a domain of the form $\mathbb{H}(\varepsilon)$. On the other hand, one can consider the Fock-Goncharov coordinates of $F(\mathcal{P}(\hbar))$ even when $\hbar$ lies outside of this domain. This leads to the following statement.

\begin{theorem}
\label{thm:meromorphic}
Suppose that $\phi$ is a quadratic differential as in the statement of Theorem~\ref{thm:intromain}. Then for each arc $j$ of the signed WKB triangulation of~$\phi$, the function 
\[
\mathcal{Y}_j:\hbar\mapsto X_j(F(\mathcal{P}(\hbar)))
\]
is a multivalued meromorphic continuation of $\mathcal{S}[e^{V_j}]$ from $\mathbb{H}(\varepsilon)$ to~$\mathbb{C}^*$, branched only at the origin.
\end{theorem}

In~\cite{GMN2}, Gaiotto, Moore, and Neitzke constructed a set of Darboux coordinates on the moduli space of rank two Higgs bundles on a Riemann surface. These coordinates are closely related to the functions $\mathcal{Y}_j$ defined in the above theorem. Indeed, there is a sense in which our construction is the ``conformal limit'' of the one in~\cite{GMN2}. This conformal limit was discussed from a physical perspective in~\cite{Gaiotto} and from a mathematical one in~\cite{DFKMMN}.

Another application of our main theorem concerns the asymptotic behavior of the monodromy map. This is closely related to the asymptotic property proved in Section~7.4 of~\cite{GMN2}. To state the result, we consider for each of the cycles $\gamma_j$ in~$\Sigma_\phi$ the period 
\[
Z_{\gamma_j}=\oint_{\gamma_j}\sqrt{\phi}
\]
where the square root is chosen so that $\Im(Z_{\gamma_j})>0$.

\begin{theorem}
\label{thm:asymptotics}
For each arc $j$ of the signed WKB triangulation of~$\phi$, we have 
\[
\mathcal{Y}_j(\hbar)\cdot\exp(Z_{\gamma_j}/\hbar)\rightarrow1
\]
as $\hbar\rightarrow0$, $\Re(\hbar)>0$.
\end{theorem}

The Darboux coordinates defined in~\cite{GMN2} were designed to give a solution of the Riemann-Hilbert problem formulated in~\cite{GMN1}. A closely related class of Riemann-Hilbert problems was recently introduced by Bridgeland~\cite{Bridgeland16} in the context of Donaldson-Thomas theory. The latter Riemann-Hilbert problems concern functions of $\hbar\in\mathbb{C}^*$ taking values in an algebraic torus. These functions are required to be piecewise continuous, having prescribed jumping behavior along a collection of rays in~$\mathbb{C}^*$, and they are required to satisfy a certain asymptotic property as $\hbar\rightarrow0$. In a forthcoming paper~\cite{Allegretti19}, we will use the results of~\cite{AllegrettiBridgeland1} to give solutions of these Riemann-Hilbert problems in a large class of examples associated to marked bordered surfaces. We will use Theorem~\ref{thm:asymptotics} to prove that our solutions have the required asymptotic property.

\subsection{Path Voros symbols and the principal cluster variety}

So far we have described Voros symbols as formal series associated to cycles in the homology group $H_1(\Sigma_\phi\setminus\widehat{P})$ where $P=\Crit(\phi)$ is the set of critical points of $\phi$ and $\widehat{P}=\pi^{-1}P$. In fact, one can also associate a Voros symbol $W_\beta$ to any path $\beta$ in the relative homology $H_1(\Sigma_\phi\setminus\widehat{P}_0,\widehat{P}_\infty)$ where $P_0=\Zer(\phi)$ is the set of zeros of~$\phi$, $P_\infty=\Pol(\phi)$ is the set of poles, and we set $\widehat{P}_0=\pi^{-1}P_0$ and $\widehat{P}_\infty=\pi^{-1}P_\infty$. These path Voros symbols have attracted attention recently in~\cite{Takei08} and~\cite{AokiTanda13}. In~\cite{IwakiNakanishi}, Iwaki and Nakanishi related the Borel sums of path Voros symbols to cluster variables in the cluster algebra with principal coefficients associated to a marked bordered surface. It would be interesting to realize the path and cycle Voros symbols as cluster coordinates on some moduli space. Such a moduli space would have twice the dimension of~$\mathcal{X}(\mathbb{S},\mathbb{M})$ and would be birational to the principal cluster variety~\cite{GHK}. Closely related spaces associated to the symplectic double cluster variety were described in~\cite{Allegretti14}, \cite{FG16}, and~\cite{Allegretti15}.

\subsection{Organization}

We begin in Section~\ref{sec:BackgroundOnWKBAnalysis} by explaining the fundamentals of exact WKB analysis following the treatment in~\cite{IwakiNakanishi}. We discuss the Schr\"odinger equation on a Riemann surface, and we define the spectral cover, WKB solutions, and Voros symbols. In Section~\ref{sec:TrajectoriesOfQuadraticDifferentials}, we discuss the foliation defined by a meromorphic quadratic differential following~\cite{BridgelandSmith} and~\cite{IwakiNakanishi}. In Section~\ref{sec:BorelResummation}, we review the theory of Borel resummation, again following~\cite{IwakiNakanishi}, and we discuss the Borel summability of the Voros symbols and WKB solutions. In Section~\ref{sec:ProjectiveStructures}, we recall the notion of a meromorphic projective structure from~\cite{AllegrettiBridgeland1}. In Section~\ref{sec:FramedLocalSystems}, we review the construction of Fock-Goncharov coordinates on the space of framed local systems and the monodromy map from~\cite{AllegrettiBridgeland1}. Finally, in Section~\ref{sec:VorosSymbolsAsClusterCoordinates}, we prove a number of connection formulas and use them to prove our main results.

\section{Background on WKB analysis}
\label{sec:BackgroundOnWKBAnalysis}

\subsection{The Schr\"odinger equation}

Let $S$ be a compact Riemann surface, and let us choose a line bundle $L$ such that $L\otimes L\cong \omega_S$ is the canonical bundle of~$S$. Then we have line bundles $\omega_S^{-1/2}$ and $\omega_S^{3/2}$ defined as the inverse and third tensor power, respectively, of $L$. For $i=1,\dots,d$, let $p_i\in S$ be distinct points, and let $m_i$ be positive integers. We form the divisor 
\[
D=\sum_im_ip_i.
\]
For any $\hbar\in\mathbb{C}$, we will consider a morphism 
\[
\mathcal{D}_\hbar:\omega_S^{-1/2}\rightarrow \omega_S^{3/2}(D)
\]
of sheaves of sections of these line bundles. Let $U\subseteq S$ be an open subset over which $L$ is trivial. In a local coordinate $z$ on~$U$, we can view a section of $\omega_S^{-1/2}$ as a holomorphic function~$y(z)$. Then the map $\mathcal{D}_\hbar$ is given by a formula 
\[
(\mathcal{D}_\hbar y)(z)=\hbar^2y''(z)-Q(z,\hbar)y(z).
\]
The function $Q(z,\hbar)$ is called the \emph{potential} and is assumed to have the form 
\[
Q(z,\hbar)=Q_0(z)+\hbar Q_1(z)+\hbar^2Q_2(z)+\dots
\]
where each $Q_n(z)$ is a meromorphic function with a pole at~$p_i$ of order at most~$m_i$, and we assume $Q_n\equiv0$ for $n\gg0$. If we take a coordinate transformation $z=z(\tilde{z})$, then we replace the potential $Q(z,\hbar)$ by 
\[
\tilde{Q}(\tilde{z},\hbar)=Q(z,\hbar)\left(\frac{dz}{d\tilde{z}}\right)^2-\frac{1}{2}\hbar^2\{z(\tilde{z}),\tilde{z}\}
\]
where $\{z(\tilde{z}),\tilde{z}\}$ is the \emph{Schwarzian derivative}:
\[
\{z(\tilde{z}),\tilde{z}\}=\frac{d}{d\tilde{z}}\left(\frac{d^2z}{d\tilde{z}^2}\bigg/\frac{dz}{d\tilde{z}}\right) - \frac{1}{2}\left(\frac{d^2z}{d\tilde{z}^2}\bigg/\frac{dz}{d\tilde{z}}\right)^2.
\]
The section considered above is represented in this new coordinate by a function $\tilde{y}(\tilde{z})=y(z)(dz/d\tilde{z})^{-1/2}$, and it is straightforward to check that $(\mathcal{D}_\hbar y)(z)(dz/d\tilde{z})^{3/2}=\hbar^2\tilde{y}''(\tilde{z})-\tilde{Q}(\tilde{z},\hbar)\tilde{y}(\tilde{z})$. Thus we get a well defined map on sections.

\begin{definition}
By a \emph{Schr\"odinger equation} on the Riemann surface~$S$, we mean an equation of the form 
\[
\mathcal{D}_\hbar y=0.
\]
\end{definition}

We will be interested in local sections of $\omega_S^{-1/2}$, defined away from the points~$p_i$, which are solutions of a Schr\"odinger equation. Note that the potential functions which define the operator $\mathcal{D}_\hbar$ are independent of the choice of square root~$L$. It follows from the transformation rule that the leading order coefficient of $\tilde{Q}(\tilde{z})$ is 
\begin{align}
\label{eqn:Q0transformation}
\tilde{Q}_0(\tilde{z})=Q_0(z)\left(\frac{dz}{d\tilde{z}}\right)^2,
\end{align}
and therefore we get a meromorphic section $\phi$ of the line bundle $\omega_S^{\otimes2}$ given locally by $\phi(z)=Q_0(z)dz^{\otimes2}$.

\begin{definition}
A meromorphic \emph{quadratic differential} is a meromorphic section of $\omega_S^{\otimes2}$.
\end{definition}

Thus we see that a Schr\"odinger equation on~$S$ determines an associated meromorphic quadratic differential. As we will see below, the properties of this quadratic differential are deeply related to the properties of the corresponding Schr\"odinger equation.

\subsection{The spectral cover}

We have just seen that there is a meromorphic quadratic differential naturally associated to a Sch\"odinger equation on a Riemann surface~$S$. We will now see how this quadratic differential, in turn, determines branched double cover of~$S$.

\begin{definition}
If $\phi$ is a quadratic differential on~$S$, then a zero or pole of $\phi$ is called a \emph{critical point}. A zero or simple pole of~$\phi$ is called a \emph{finite critical point}, and any other critical point is called an \emph{infinite critical point}. A quadratic differential is said to be \emph{complete} if it has no simple poles.
\end{definition}

We will denote the set of all critical points of $\phi$ by $\Crit(\phi)$. We will denote by $\Zer(\phi)$ and $\Pol(\phi)$ the subsets of zeros and poles, and by $\Crit_{<\infty}(\phi)$ and $\Crit_\infty(\phi)$ the subsets of finite and infinite critical points, respectively. In this paper, we will restrict attention to the following class of differentials.

\begin{definition}[\cite{BridgelandSmith}, Definition~2.1]
A meromorphic quadratic differential~$\phi$ on a Riemann surface~$S$ is called a \emph{Gaiotto-Moore-Neitzke (GMN) differential} if 
\begin{enumerate}
\item $\phi$ has simple zeros.
\item $\phi$ has at least one pole.
\item $\phi$ has at least one finite critical point.
\end{enumerate}
\end{definition}

Let $\phi$ be a GMN differential on a compact Riemann surface~$S$ with poles of order $m_i$ at the points $p_i\in S$. We can alternatively view $\phi$ as a holomorphic section 
\[
s_\phi\in H^0(S,\omega_S(E)^{\otimes2}), \quad E=\sum_i\left\lceil\frac{m_i}{2}\right\rceil p_i
\]
with simple zeros at both the zeros and odd order poles of~$\phi$.

\begin{definition}
The \emph{spectral cover} is defined as 
\[
\Sigma_\phi=\{(p,\psi(p)):p\in S,\psi(p)\in F_p,\psi(p)\otimes\psi(p)=s_\phi(p)\}\subseteq F.
\]
where $F$ denotes the total space of the line bundle $\omega_S(E)$.
\end{definition}

The space $\Sigma_\phi$ is a manifold because $\phi$ is assumed to have simple zeros. The natural projection $\pi:\Sigma_\phi\rightarrow S$ is a double cover branched at the simple zeros and odd order poles of~$\phi$. We will denote by $z^*$ the image of $z\in\Sigma_\phi$ under the covering involution $\Sigma_\phi\rightarrow\Sigma_\phi$.

\subsection{WKB solutions}

We will now define a pair of formal solutions of the Schr\"odinger equation. In a local coordinate~$z$, this is equation is 
\begin{align}
\label{eqn:schrodinger}
\hbar^2y''(z,\hbar)-Q(z,\hbar)y(z,\hbar)=0
\end{align}
where $Q(z,\hbar)=Q_0(z)+\hbar Q_1(z)+\dots$ is as above. To solve it, we consider a family $\{S_n(z)\}_{n\geq-1}$ of functions defined by the initial condition 
\[
S_{-1}^2=Q_0(z)
\]
and the recursion relations 
\[
2S_{-1}S_{n+1}+\sum_{\substack{n_1+n_2=n \\ 0\leq n_j\leq n}}S_{n_1}S_{n_2}+\frac{dS_n}{dz}=Q_{n+2}(z)
\]
for $n\geq-1$. There are two families $\{S_n^{(+)}(z)\}_{n\geq-1}$ and $\{S_n^{(-)}(z)\}_{n\geq-1}$ of functions satisfying these recursion relations depending on the choice of root $S_{-1}^{(\pm)}(z)=\pm\sqrt{Q_0(z)}$ for the initial condition. We consider the formal series 
\[
S^{(\pm)}(z,\hbar)=\sum_{n=-1}^\infty\hbar^nS_n^{(\pm)}(z)
\]
and the ``odd part'' 
\[
S_{\text{odd}}(z,\hbar)=\frac{1}{2}\left(S^{(+)}(z,\hbar)-S^{(-)}(z,\hbar)\right).
\]
The following proposition shows how this expression transforms under the change of coordinates $z=z(\tilde{z})$.

\begin{proposition}[\cite{IwakiNakanishi}, Proposition~2.7(b)]
\label{prop:transformSodd}
Let $\tilde{Q}_0(\tilde{z})$ be the leading term in the transformed potential function given by~\eqref{eqn:Q0transformation}, and let $\sqrt{\tilde{Q}_0(\tilde{z})}$ be a square root of this function, chosen so that 
\[
\sqrt{\tilde{Q}_0(\tilde{z})}=\sqrt{Q_0(z)}\frac{dz}{d\tilde{z}}.
\]
Then the formal series $\tilde{S}_{\text{odd}}(\tilde{z},\hbar)$ constructed from $\sqrt{\tilde{Q}_0(\tilde{z})}$ by the above recursion relations satisfies 
\[
\tilde{S}_{\text{odd}}(\tilde{z},\hbar)=S_{\text{odd}}(z,\hbar)\frac{dz}{d\tilde{z}}.
\]
\end{proposition}

More geometrically, Proposition~\ref{prop:transformSodd} says that there is a well defined (formal series valued) meromorphic 1-form on~$\Sigma_\phi$ given locally by the expression $S_{\text{odd}}(z,\hbar)dz$. In exact WKB~analysis, one is interested in formal series defined by integrating this 1-form along paths in the spectral cover.

\begin{definition}[\cite{IwakiNakanishi}, Definition~2.9]
\label{def:WKBsolutions}
The \emph{WKB solutions} are the formal solutions 
\[
\psi_{\pm}(z,\hbar)=\frac{1}{\sqrt{S_{\text{odd}}(z,\hbar)}}\exp\left(\pm\int^zS_{\text{odd}}(\zeta,\hbar)d\zeta\right)
\]
of equation~\eqref{eqn:schrodinger}.
\end{definition}

The integral appearing in this definition is defined by integrating the coefficient of each power of $\hbar$ in the formal series $S_{\text{odd}}(\zeta,\hbar)$ termwise. To get a well defined formal solution, we must also specify the lower limit of integration. We would like to take this lower limit to be a point of~$\pi^{-1}\Pol(\phi)$, but we are unable to do so because the 1-form $S_{\text{odd}}(z,\hbar)dz$ is not integrable near a pole. To overcome this difficulty, we impose an assumption on the coefficients of the potential function.

\Needspace*{2\baselineskip}
\begin{assumption}[\cite{IwakiNakanishi}, Assumption~2.5]\mbox{}
\label{assumption2.5}
\begin{enumerate}
\item If $p\in S$ is a pole of $Q_n(z)$ for some $n\geq1$, then $p\in\Pol(\phi)$.
\item If $\phi$ has a pole $p$ of order $m\geq3$, then 
\[
\left(\text{order of $Q_n(z)$ at $p$}\right)<1+\frac{m}{2} \quad \text{for all $n\geq1$}.
\]
\item If $\phi$ has a pole at~$p$ of order $m=2$, then 
\begin{enumerate}
\item $Q_n(z)$ has at most a simple pole at $p$ for all $n\geq1$ except for $n=2$.
\item $Q_2(z)$ has a double pole at~$p$ and satisfies 
\[
Q_2(z)=-\frac{1}{4z^2}\left(1+O(z)\right) \quad \text{as $z\rightarrow0$}
\]
where $z$ is a local coordinate around $p$ such that $z(p)=0$.
\end{enumerate}
\end{enumerate}
\end{assumption}

Under this assumption, we introduce the expression 
\[
S_{\text{odd}}^{\text{reg}}(z,\hbar)dz=\left(S_{\text{odd}}(z,\hbar)-\frac{1}{\hbar}\sqrt{Q_0(z)}\right)dz,
\]
which defines a 1-form on~$\Sigma_\phi$.

\begin{theorem}[\cite{IwakiNakanishi}, Proposition~2.8]
Suppose the coefficients $\{Q_n(z)\}$ of the potential function satisfy Assumption~\ref{assumption2.5}. Then for any point $p\in\Pol(\phi)$, the formal series valued 1-form $S_{\text{odd}}^{\text{reg}}(z,\eta)dz$ is integrable at~$p$.
\end{theorem}

If Assumption~\ref{assumption2.5} is satisfied, we can therefore employ following scheme to define the integral in Definition~\ref{def:WKBsolutions}.

\begin{definition}
\label{def:normalizedWKB}
Let $p\in\Pol(\phi)$. Then the WKB solution \emph{normalized at $p$} is the expression 
\[
\psi_{\pm}(z,\hbar)=\frac{1}{\sqrt{S_{\text{odd}}(z,\hbar)}}\exp\left(\pm\left(\frac{1}{\hbar}\int_a^z\sqrt{Q_0(\zeta)}d\zeta+\int_p^zS_{\text{odd}}^{\text{reg}}(\zeta,\hbar)d\zeta\right)\right)
\]
where $a$ is any zero of $\phi$ independent of~$p$. We will write this as $\psi_{\pm}^{(p,a)}(z,\hbar)$ when we wish to emphasize the choice of~$p$ and~$a$.
\end{definition}

Thus we have a pair of well-normalized formal solutions of~\eqref{eqn:schrodinger}. Each WKB solution can be expanded as a formal power series in $\hbar$ multiplied by an additional factor:
\[
\psi_{\pm}(z,\hbar)=\exp\left(\pm\frac{1}{\hbar}\int^z\sqrt{Q_0(\zeta)}d\zeta\right)\hbar^{1/2}\sum_{k=0}^\infty\hbar^k\psi_{\pm,k}(z).
\]
It is known that the series in this expression is divergent in general, and therefore, in order to get a genuine analytic solution of~\eqref{eqn:schrodinger}, we must take the Borel resummation. We will revisit this issue after reviewing some properties of the foliation induced by a quadratic differential.

\subsection{Voros symbols}

For a fixed quadratic differential $\phi$, let us simplify notation by writing $\Sigma=\Sigma_\phi$, $P_0=\Zer(\phi)$, $P_\infty=\Pol(\phi)$, and $P=\Crit(\phi)$. Let $\widehat{P}$, $\widehat{P}_0$, and~$\widehat{P}_\infty$ denote the preimages of~$P$, $P_0$ and~$P_\infty$, respectively, in the the cover~$\Sigma$. Consider the subsets of the homology groups $H_1(\Sigma\setminus\widehat{P}_0,\widehat{P}_\infty)$ and $H_1(\Sigma\setminus\widehat{P})$ consisting of elements that are invariant under the covering involution:
\begin{align*}
\Sym(H_1(\Sigma\setminus\widehat{P}_0,\widehat{P}_\infty)) &= \{\beta\in H_1(\Sigma\setminus\widehat{P}_0,\widehat{P}_\infty):\beta^*=\beta\}, \\
\Sym(H_1(\Sigma\setminus\widehat{P})) &= \{\gamma\in H_1(\Sigma\setminus\widehat{P}):\gamma^*=\gamma\}.
\end{align*}
We denote the quotients of the homology groups by these subgroups by 
\begin{align*}
\mathcal{H}^\vee &= H_1(\Sigma\setminus\widehat{P}_0,\widehat{P}_\infty)/\Sym(H_1(\Sigma\setminus\widehat{P}_0,\widehat{P}_\infty)), \\
\mathcal{H} &= H_1(\Sigma\setminus\widehat{P})/\Sym(H_1(\Sigma\setminus\widehat{P})).
\end{align*}
Thus we have $\eta=-\eta^*$ for $\eta$ in either of these quotients.

We can now define one of the central concepts of this paper.

\Needspace*{2\baselineskip}
\begin{definition}[\cite{IwakiNakanishi}, Definition~3.1]\mbox{}
\begin{enumerate}
\item For each path $\beta\in\mathcal{H}^\vee$, we write 
\[
W_\beta=\int_\beta S_{\text{odd}}^{\text{reg}}(z,\hbar)dz.
\]
Then the formal series $e^{W_\beta}$ is called the \emph{Voros symbol} for the path~$\beta$.
\item For each cycle $\gamma\in\mathcal{H}$, we write 
\[
V_\gamma=\oint_\gamma S_{\text{odd}}(z,\hbar)dz.
\]
Then the formal series $e^{V_\gamma}$ is called the \emph{Voros symbol} for the cycle~$\gamma$.
\end{enumerate}
\end{definition}

The Voros symbols are well defined because of the identities $S_{\text{odd}}(z^*,\hbar)=-S_{\text{odd}}(z,\hbar)$ and $S_{\text{odd}}^{\text{reg}}(z^*,\hbar)=-S_{\text{odd}}^{\text{reg}}(z,\hbar)$.

\section{Trajectories of quadratic differentials}
\label{sec:TrajectoriesOfQuadraticDifferentials}

\subsection{The horizontal foliation}

As mentioned in the introduction, a quadratic differential on a Riemann surface~$S$ determines an associated foliation of~$S$. This section will be devoted to reviewing the geometry of this foliation.

Let $\phi$ be a quadratic differential on~$S$. In a neighborhood of any point which is not a critical point of~$\phi$, there is a distinguished local coordinate $w$, unique up to transformations of the form $w\mapsto\pm w+\text{constant}$, such that in this local coordinate the quadratic differential can be written 
\[
\phi(w)=dw^{\otimes2}.
\]
Indeed, if we have $\phi(z)=\varphi(z)dz^{\otimes2}$ for some local coordinate $z$ away from the critical points, then $w$ is given by 
\[
w(z)=\int^z\sqrt{\varphi(z)}dz
\]
for some choice of square root of $\varphi(z)$.

\begin{definition}
If $\phi$ is a quadratic differential on~$S$, then a \emph{horizontal trajectory} of~$\phi$ is a curve in $S\setminus\Crit(\phi)$ given by $\Im(w)=\text{constant}$ where $w$ is the distinguished local coordinate. The \emph{horizontal foliation} is the foliation of $S\setminus\Crit(\phi)$ by horizontal trajectories.
\end{definition}

To understand the geometry of the horizontal trajectories, we first consider their behavior near the critical points of~$\phi$. Near a zero of order $k\geq1$, it is known that the horizontal trajectories form a $(k+2)$-pronged singularity as illustrated below for $k=1,2$.
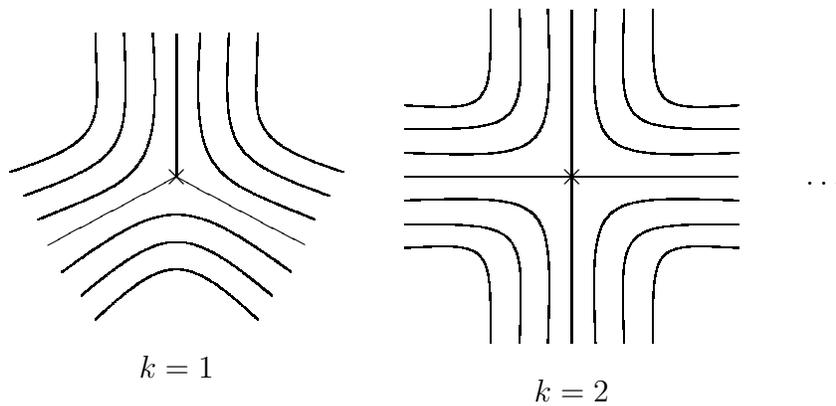
\[
\xy /l3pc/:
(1,0)*{}="O";  
(-0.35,0.72)*{}="U";  
(-0.75,-0.05)*{}="X1";
(-0.6,0.2)*{}="X2"; 
(-0.45,0.45)*{}="X3"; 
(-0.2,1)*{}="X4";  
(0,1.25)*{}="X5";  
(0.15,1.5)*{}="X6"; 
(1.85,1.5)*{}="Y1";
(2,1.25)*{}="Y2";
(2.2,1)*{}="Y3";
(2.45,0.45)*{}="Y4"; 
(2.6,0.2)*{}="Y5";
(2.75,-0.05)*{}="Y6"; 
(1.85,-1.5)*{}="Z1";
(1.55,-1.5)*{}="Z2"; 
(1.25,-1.5)*{}="Z3";
(0.75,-1.5)*{}="Z4";
(0.45,-1.5)*{}="Z5";
(0.15,-1.5)*{}="Z6";
(2.35,0.72)*{}="V";  
(1,-1.5)*{}="W"; 
"O";"U" **\dir{-};  
"O";"V" **\dir{-}; 
"O";"W" **\dir{-}; 
"X4";"Y3" **\crv{(0.9,0.2) & (1.1,0.2)};
"X5";"Y2" **\crv{(0.9,0.5) & (1.1,0.5)}; 
"X6";"Y1" **\crv{(0.9,0.8) & (1.1,0.8)};
"Y4";"Z3" **\crv{(1.35,0) & (1.15,0)};
"Y5";"Z2" **\crv{(1.5,-0.2) & (1.5,-0.3)};
"Y6";"Z1" **\crv{(1.65,-0.4) & (1.85,-0.6)};
"Z4";"X3" **\crv{(0.85,0) & (0.65,0)};
"Z5";"X2" **\crv{(0.5,-0.3) & (0.5,-0.2)};
"Z6";"X1" **\crv{(0.15,-0.6) & (0.35,-0.4)};
(1,0)*{\times};
(1,2)*{k=1};
\endxy
\qquad 
\xy /l3pc/:
(1,0)*{}="O";  
(1,-1.75)*{}="T"; 
(-0.75,0)*{}="U"; 
(1,1.75)*{}="V";  
(2.75,0)*{}="W"; 
(-0.75,-0.75)*{}="U1";
(-0.75,-0.5)*{}="U2";
(-0.75,-0.25)*{}="U3"; 
(-0.75,0.25)*{}="U4"; 
(-0.75,0.5)*{}="U5";
(-0.75,0.75)*{}="U6";
(0.15,1.75)*{}="V1";  
(0.45,1.75)*{}="V2";
(0.75,1.75)*{}="V3";
(1.25,1.75)*{}="V4";  
(1.55,1.75)*{}="V5"; 
(1.85,1.75)*{}="V6"; 
(2.75,0.75)*{}="W1";
(2.75,0.5)*{}="W2";
(2.75,0.25)*{}="W3"; 
(2.75,-0.25)*{}="W4";
(2.75,-0.5)*{}="W5";
(2.75,-0.75)*{}="W6";
(1.85,-1.75)*{}="T1";
(1.55,-1.75)*{}="T2";
(1.25,-1.75)*{}="T3";
(0.75,-1.75)*{}="T4";
(0.45,-1.75)*{}="T5"; 
(0.15,-1.75)*{}="T6";
"O";"T" **\dir{-};   
"O";"U" **\dir{-}; 
"O";"V" **\dir{-};  
"O";"W" **\dir{-}; 
"U4";"V3" **\crv{(0.8,0.2) & (0.8,0.2)};
"U5";"V2" **\crv{(0.5,0.5) & (0.5,0.5)};
"U6";"V1" **\crv{(0.15,0.7) & (0.15,0.7)}; 
"V4";"W3" **\crv{(1.2,0.2) & (1.2,0.2)};
"V5";"W2" **\crv{(1.5,0.5) & (1.5,0.5)}; 
"V6";"W1" **\crv{(1.85,0.7) & (1.85,0.7)};
"W4";"T3" **\crv{(1.2,-0.2) & (1.2,-0.2)};
"W5";"T2" **\crv{(1.5,-0.5) & (1.5,-0.5)}; 
"W6";"T1" **\crv{(1.85,-0.7) & (1.85,-0.7)};
"T4";"U3" **\crv{(0.8,-0.2) & (0.8,-0.2)}; 
"T5";"U2" **\crv{(0.5,-0.5) & (0.5,-0.5)};
"T6";"U1" **\crv{(0.15,-0.7) & (0.15,-0.7)}; 
(1,0)*{\times};
(1,2.25)*{k=2};
\endxy
\qquad
\dots
\]

On the other hand, near a pole of order~2, there is a local coordinate $t$ such that the differential can be written 
\[
\phi(t)=\frac{r}{t^2}dt^{\otimes2}
\]
for some well defined constant $r\in\mathbb{C}^*$.

\begin{definition}
The \emph{residue} of $\phi$ at $p$ is the quantity 
\[
\Res_p(\phi)=\pm4\pi i\sqrt{r},
\]
which is well defined up to a sign.
\end{definition}

Near the double pole $p$, the horizontal foliation can exhibit three possible behaviors in the $t$-plane depending on the value of the residue at~$p$: 
\begin{enumerate}
\item If $\Res_p(\phi)\in\mathbb{R}$, then the horizontal trajectories are concentric circles centered on the pole.
\item If $\Res_p(\phi)\in i\mathbb{R}$, then the horizontal trajectories are radial arcs emanating from the pole.
\item If $\Res_p(\phi)\not\in\mathbb{R}\cup i\mathbb{R}$, then the horizontal trajectories are logarithmic spirals that wrap around the pole.
\end{enumerate}
The diagrams below illustrate the three types of foliations.
\[
\xy /l1.5pc/:
(1,-3)*\xycircle(3,3){-};
(1,-2)*\xycircle(2,2){-};
(1,-1)*\xycircle(1,1){-};
(1,0)*{\bullet};
(1,4)*{\Res_p(\phi)\in\mathbb{R}};
\endxy
\qquad
\qquad
\xy /l1.5pc/:
{\xypolygon12"A"{~:{(2,2):}~>{}}};
{(1,0)\PATH~={**@{-}}'"A1"};
{(1,0)\PATH~={**@{-}}'"A2"};
{(1,0)\PATH~={**@{-}}'"A3"};
{(1,0)\PATH~={**@{-}}'"A4"};
{(1,0)\PATH~={**@{-}}'"A5"};
{(1,0)\PATH~={**@{-}}'"A6"};
{(1,0)\PATH~={**@{-}}'"A7"};
{(1,0)\PATH~={**@{-}}'"A8"};
{(1,0)\PATH~={**@{-}}'"A9"};
{(1,0)\PATH~={**@{-}}'"A10"};
{(1,0)\PATH~={**@{-}}'"A11"};
{(1,0)\PATH~={**@{-}}'"A12"};
(1,0)*{\bullet};
(1,4)*{\Res_p(\phi)\in i\mathbb{R}};
\endxy
\qquad
\xy /l1.5pc/:
(1,0)*{\bullet};
(1.13,0);(-2,0) **\crv{(1.5,0.5) & (0,1.5)};
(1,-0.13);(1,3) **\crv{(1.5,-0.5) & (2.5,1)};
(0.87,0);(4,0) **\crv{(0.5,-0.5) & (2,-1.5)};
(1,0.13);(1,-3) **\crv{(0.5,0.5) & (-0.5,-1)};
(1,4)*{\Res_p(\phi)\not\in\mathbb{R}\cup i\mathbb{R}};
\endxy
\]

Finally, if $p\in\Pol(\phi)$ is a pole of order $m\geq3$, then there is a neighborhood $U$ of $p$ and a collection of $m-2$ distinguished tangent directions $v_i$ at $p$ such that any horizontal trajectory that enters $U$ eventually tends to $p$ and is asymptotic to one of the $v_i$. We illustrate this below for $m=5,6$.
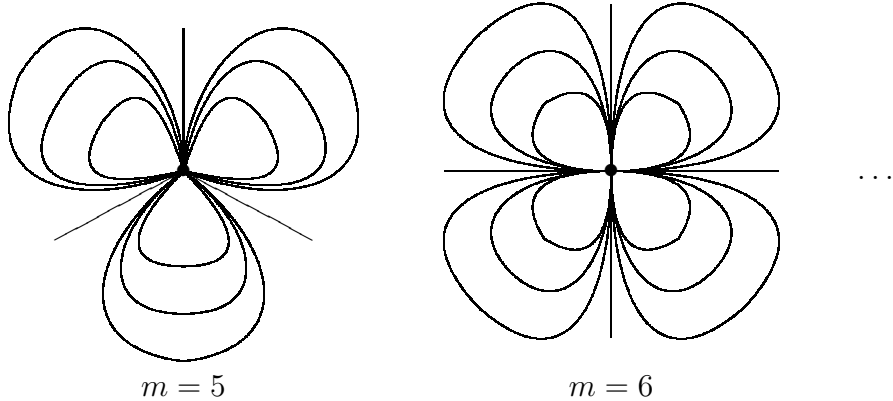
\[
\xy /l3pc/:
{\xypolygon3"T"{~:{(2,0):}~>{}}},
{\xypolygon3"S"{~:{(1.5,0):}~>{}}},
{\xypolygon3"R"{~:{(1,0):}~>{}}},
(1,0)*{}="O"; 
(-0.35,0.72)*{}="U"; 
(2.35,0.72)*{}="V"; 
(1,-1.5)*{}="W"; 
"O";"U" **\dir{-}; 
"O";"V" **\dir{-}; 
"O";"W" **\dir{-}; 
"O";"T1" **\crv{(2,0.75) & (2.25,1.85)};
"O";"T1" **\crv{(0,0.75) & (-0.25,1.85)};
"O";"T2" **\crv{(0,0.4) & (-1.2,0.25)};
"O";"T2" **\crv{(1,-1) & (0,-2.2)};
"O";"T3" **\crv{(1,-1) & (2,-2.2)};
"O";"T3" **\crv{(2,0.4) & (3.2,0.25)};
"O";"S1" **\crv{(1.75,0.56) & (2,1.5)};
"O";"S1" **\crv{(0.25,0.56) & (0,1.5)};
"O";"S2" **\crv{(0,0.3) & (-0.9,0.19)};
"O";"S2" **\crv{(0.9,-0.8) & (0.3,-1.7)};
"O";"S3" **\crv{(2,0.3) & (2.9,0.19)};
"O";"S3" **\crv{(1.1,-0.8) & (1.7,-1.7)};
"O";"R1" **\crv{(1.5,0.5) & (1.75,1)};
"O";"R1" **\crv{(0.5,0.5) & (0.25,1)};
"O";"R2" **\crv{(0.5,0.1) & (-0.3,0.25)};
"O";"R2" **\crv{(0.75,-0.8) & (0.5,-1)};
"O";"R3" **\crv{(1.5,0.1) & (2.3,0.25)};
"O";"R3" **\crv{(1.25,-0.8) & (1.5,-1)};
(1,0)*{\bullet};
(1,2.25)*{m=5};
\endxy
\qquad
\xy /l3pc/:
{\xypolygon4"A"{~:{(2,0):}~>{}}},
{\xypolygon4"B"{~:{(1.5,0):}~>{}}},
{\xypolygon4"C"{~:{(1,0):}~>{}}},
(1,0)*{}="O"; 
(1,-1.75)*{}="T"; 
(-0.75,0)*{}="U"; 
(1,1.75)*{}="V"; 
(2.75,0)*{}="W"; 
"O";"T" **\dir{-}; 
"O";"U" **\dir{-}; 
"O";"V" **\dir{-}; 
"O";"W" **\dir{-}; 
"O";"A1" **\crv{(1,1.5) & (1.5,2.25)};
"O";"A1" **\crv{(2.5,0) & (3.25,0.5)};
"O";"A2" **\crv{(1,1.5) & (0.5,2.25)};
"O";"A2" **\crv{(-0.5,0) & (-1.25,0.5)};
"O";"A3" **\crv{(1,-1.5) & (0.5,-2.25)};
"O";"A3" **\crv{(-0.5,0) & (-1.25,-0.5)};
"O";"A4" **\crv{(1,-1.5) & (1.5,-2.25)};
"O";"A4" **\crv{(2.5,0) & (3.25,-0.5)};
"O";"B1" **\crv{(1,1) & (1.5,1.6)};
"O";"B1" **\crv{(2,0) & (2.6,0.5)};
"O";"B2" **\crv{(1,1) & (0.5,1.6)};
"O";"B2" **\crv{(0,0) & (-0.6,0.5)};
"O";"B3" **\crv{(1,-1) & (0.5,-1.6)};
"O";"B3" **\crv{(0,0) & (-0.6,-0.5)};
"O";"B4" **\crv{(1,-1) & (1.5,-1.6)};
"O";"B4" **\crv{(2,0) & (2.6,-0.5)};
"O";"C1" **\crv{(1,0.75) & (1.25,1)};
"O";"C1" **\crv{(1.75,0) & (2,0.25)};
"O";"C2" **\crv{(1,0.75) & (0.75,1)};
"O";"C2" **\crv{(0.25,0) & (0,0.25)};
"O";"C3" **\crv{(1,-0.75) & (0.75,-1)};
"O";"C3" **\crv{(0.25,0) & (0,-0.25)};
"O";"C4" **\crv{(1,-0.75) & (1.25,-1)};
"O";"C4" **\crv{(1.75,0) & (2,-0.25)};
(1,0)*{\bullet};
(1,2.25)*{m=6};
\endxy
\qquad
\dots
\]

To understand the global geometry of the horizontal foliation, one needs to consider five types of horizontal trajectories.

\Needspace*{2\baselineskip}
\begin{definition}\mbox{}
\begin{enumerate}
\item A \emph{saddle trajectory} is a horizontal trajectory which connects finite critical points of~$\phi$.
\item A \emph{separating trajectory} is a horizontal trajectory which connects a finite and an infinite critical point of~$\phi$.
\item A \emph{generic trajectory} is a horizontal trajectory which connects infinite critical points of~$\phi$.
\item A \emph{closed trajectory} is a horizontal trajectory which is a simple closed curve in $S\setminus\Crit(\phi)$.
\item A \emph{recurrent trajectory} is a horizontal trajectory which is recurrent in at least one direction.
\end{enumerate}
\end{definition}

As shown in Sections~9--11 of~\cite{Strebel}, every horizontal trajectory belongs to one of these five categories.

\subsection{Saddle-free differentials}

Recall that a differential is said to be \emph{saddle-free} if the associated horizontal foliation has no saddle trajectories.

\begin{lemma}[\cite{BridgelandSmith}, Lemma~3.1, \cite{GMN2}, Section~6.3]
\label{lem:saddlefree}
If $\phi$ is a saddle-free GMN differential such that $\Crit_\infty(\phi)\neq\emptyset$, then the associated horizontal foliation has no closed or recurrent trajectories.
\end{lemma}

We have seen that there are only finitely many horizontal trajectories incident to any zero of a quadratic differential. Thus, if $\phi$ satisfies the hypotheses of the above lemma, there are finitely many separating trajectories which divide the surface~$S$ into regions foliated by generic trajectories.

\Needspace*{2\baselineskip}
\begin{definition} Assume $\phi$ is a saddle-free GMN differential such that $\Crit_\infty(\phi)\neq\emptyset$.
\begin{enumerate}
\item A \emph{horizontal strip} is a connected component of the complement of the separating trajectories in~$S$ which can be mapped isomorphically by the distinguished local coordinate to a subset of~$\mathbb{C}$ of the form 
\[
\{w\in\mathbb{C}:a<\Im(w)<b\}.
\]
The trajectories in a horizontal strip are generic, connecting two (not necessarily distinct) poles.

\item A \emph{half plane} is a connected component of the complement of the separating trajectories in~$S$ which can be mapped isomorphically by the distinguished local coordinate to a subset of~$\mathbb{C}$ of the form 
\[
\{w\in\mathbb{C}:\Im(w)>0\}.
\]
The trajectories in a half plane are generic, connecting a fixed pole of order $>2$ to itself.
\end{enumerate}
\end{definition}

If $\phi$ satisfies the hypotheses of Lemma~\ref{lem:saddlefree}, then after removing the finitely many separating trajectories, we are left with an open subsurface which is a disjoint union of horizontal strips and half planes. Each component of the boundary of a horizontal strip in this decomposition contains exactly one finite critical point. Suppose these are both zeros. Then we call the horizontal strip \emph{regular} if these zeros are distinct in~$S$ and \emph{degenerate} if they coincide.

\subsection{The WKB triangulation}

The connection between exact WKB analysis and cluster algebras arises because any complete saddle-free GMN differential determines a corresponding marked bordered surface equipped with an ideal triangulation. In this section, we will review this construction along with the basic theory of triangulated surfaces.

\begin{definition}
A \emph{marked bordered surface} is a pair $(\mathbb{S},\mathbb{M})$ where $\mathbb{S}$ is a compact, connected, oriented, smooth surface with (possibly empty) boundary, and $\mathbb{M}$ is a nonempty finite set of marked points on~$\mathbb{S}$ such that each boundary component of~$\mathbb{S}$ contains at least one marked point. A marked point in the interior of~$\mathbb{S}$ is called a \emph{puncture}, and we write $\mathbb{P}$ for the set of all punctures.
\end{definition}

If $\phi$ is a meromorphic quadratic differential on~$S$ with at least one pole, then we can define a surface $\mathbb{S}$ by performing a oriented real blowup of~$S$ at each pole of~$\phi$ of order~$\geq3$. We define a finite set $\mathbb{M}\subseteq\mathbb{S}$ consisting of punctures given by the poles of~$\phi$ of order~$\leq2$ and points on~$\partial\mathbb{S}$ given by the distinguished tangent directions described above. In this way, we get a marked bordered surface associated to~$\phi$.

\begin{definition}
Let $(\mathbb{S},\mathbb{M})$ be a marked bordered surface. An \emph{arc} in $(\mathbb{S},\mathbb{M})$ is a smooth path $\gamma$ in $\mathbb{S}$ connecting points of $\mathbb{M}$ whose interior lies in the interior of $\mathbb{S}$ and which has no self-intersections in its interior. We also require that $\gamma$ is not homotopic, relative to its endpoints, to a single point or to a path in $\partial\mathbb{S}$ whose interior contains no marked points. Two arcs are considered to be equivalent if they are related by a homotopy through such arcs. A path that connects two marked points and lies entirely on the boundary of~$\mathbb{S}$ without passing through a third marked point is called a \emph{boundary segment}.
\end{definition}

\begin{definition}
Two arcs are said to be \emph{compatible} if there exist curves in their respective equivalence classes that do not intersect in the interior of $\mathbb{S}$. A maximal collection of pairwise compatible arcs is called an \emph{ideal triangulation} of~$(\mathbb{S},\mathbb{M})$. The arcs of an ideal triangulation cut $\mathbb{S}$ into regions called \emph{ideal triangles}.
\end{definition}

Note that an ideal triangle may have fewer than three distinct sides. In this case, it is called a \emph{self-folded triangle}, and it looks like the diagram below.
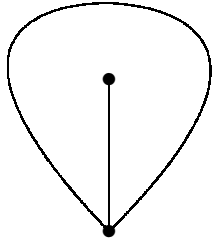
\[
\xy /l0.6pc/:
(0,-4)*{}="N"; 
(0,8)*{}="S"; 
(0,0);"S" **\dir{-}; 
(0,0)*{\bullet}; 
"S"*{\bullet}; 
"S";"N" **\crv{(-12,-4) & (0,-4)}; 
"S";"N" **\crv{(12,-4.5) & (0,-4)}; 
\endxy
\]

If $\phi$ is a complete and saddle-free differential on a compact Riemann surface~$S$, we have seen that the separating trajectories divide the surface $S$ into finitely many horizontal strips and half planes. If we choose a single generic trajectory from each of the horizontal strips, we get a collection of arcs which define an ideal triangulation of the associated marked bordered surface.

\begin{definition}
The ideal triangulation constructed from a quadratic differential in this way is called the \emph{WKB triangulation}.
\end{definition}

As we will see below, there are Voros symbols and cluster coordinates naturally indexed by the edges of the WKB triangulation.

\subsection{Foliation of the spectral cover}
\label{sec:FoliationOfTheSpectralCover}

Suppose $\phi$ is a GMN differential. Then the inverse image of the horizontal foliation of $S\setminus\Crit(\phi)$ under the covering map~$\pi$ is a foliation of $\Sigma_\phi\setminus\pi^{-1}\Crit(\phi)$.

\begin{lemma}
Let $\beta$ be a generic trajectory of the quadratic differential $\phi$. Then the real part of the distinguished local coordinate $w(z)$ is increasing or decreasing along~$\beta$.
\end{lemma}

\begin{proof}
Suppose the trajectory $\beta$ is given by a map $\gamma:(0,1)\rightarrow S$ with $\gamma'(t)\neq0$ for all $t\in(0,1)$. By definition of a trajectory, we know that the imaginary part $\Im w$ is constant along~$\beta$, that is, $\frac{d}{dt}\Im w(\gamma(t))=0$ for all $t\in(0,1)$. If we also have $\frac{d}{dt}\Re w(\gamma(t))=0$ at some~$t=t_0$, then 
\[
w'(\gamma(t_0))\gamma'(t_0)=\frac{d}{dt}\Re w(\gamma(t))\big\rvert_{t=t_0}+i\frac{d}{dz}\Im w(\gamma(t))\big\rvert_{t=t_0}=0,
\]
so $w'(\gamma(t_0))=0$. But then $\phi(z)=(w'(z))^2dz^{\otimes2}$ has a zero at~$z=\gamma(t_0)$, contradicting the fact that $\beta$ is a generic trajectory. Hence $\frac{d}{dt}\Re w(\gamma(t))\neq0$ for all $t\in(0,1)$, and $\Re w$ is increasing or decreasing along~$\beta$.
\end{proof}

Note that the function $w(z)=\int^z\sqrt{\varphi}$ is well defined on the cover $\Sigma_\phi$. Thus we see that each leaf in the foliation of $S\setminus\Crit(\phi)$ has a natural orientation defined by the following rule: The real part of the function $w(z)$ increases along a leaf of the foliation in the positive direction.

In the following, it will be useful to have a way of drawing pictures of the Riemann surface~$\Sigma_\phi$. To this end, let us choose branch cuts on the surface~$S$. A branch cut is a path on~$S$ connecting two branch points for the covering map~$\pi$. If we choose a collection of branch cuts so that each branch point is an endpoint of some branch cut, then we can choose an embedding $S\rightarrow\Sigma_\phi$ which is piecewise continuous and has discontinuities along the branch cuts. We typically represent the image of this embedding by drawing its projection to~$S$ and indicating the branch cuts by wavy lines on~$S$. The image of our embedding in~$\Sigma_\phi$ is called the \emph{first sheet}, and its complement in~$\Sigma_\phi$ is called the \emph{second sheet}. When drawing pictures of the cover $\Sigma_\phi$, we assign the symbols $\oplus$ and~$\ominus$ to the preimages of the poles in such a way that any leaf of the foliation is oriented from~$\ominus$ to~$\oplus$.

With these conventions, we can describe an important collection of elements of $\mathcal{H}$ and~$\mathcal{H}^\vee$. To define them, suppose $\phi$ is a complete saddle-free differential, and let $j$ be a generic trajectory. If $j$ is contained in a regular horizontal strip~$D_j$ which does not surround a degenerate horizontal strip, then we define $\gamma_j\in\mathcal{H}$ and $\beta_j\in\mathcal{H}^\vee$ as in the pictures
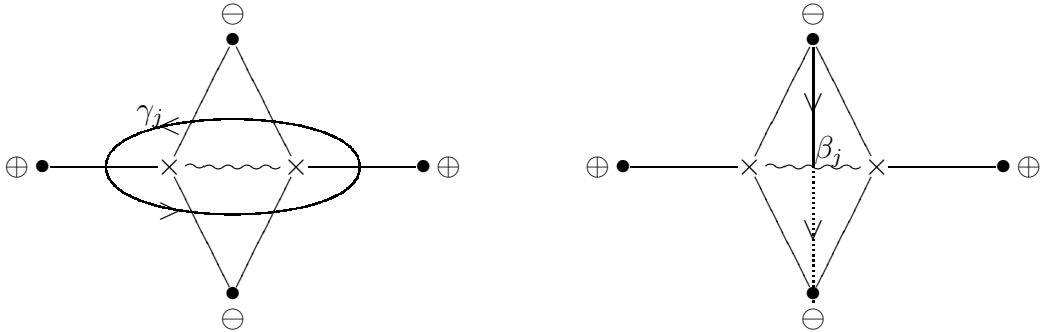
\[
\xy /l2pc/:
(0,-2)*{\bullet}="21";
(3,0)*{\bullet}="12";
(1,0)*{\times}="22";
(-1,0)*{\times}="32";
(-3,0)*{\bullet}="42";
(0,2)*{\bullet}="23";
(2,0)*{}="1";
(-2,0)*{}="2";
{"22"\PATH~={**@{~}}'"32"},
{"12"\PATH~={**@{-}}'"22"},
{"32"\PATH~={**@{-}}'"42"},
{"21"\PATH~={**@{-}}'"22"},
{"21"\PATH~={**@{-}}'"32"},
{"22"\PATH~={**@{-}}'"23"},
{"32"\PATH~={**@{-}}'"23"},
(0,-2.4)*{\ominus};
(0,2.4)*{\ominus};
(-3.4,0)*{\oplus};
(3.4,0)*{\oplus};
"1";"2" **\crv{(2,-1) & (-2,-1)}; 
"1";"2" **\crv{(2,1) & (-2,1)}; 
(1,-0.65)*{<};
(1,0.7)*{>};
(1.3,-0.8)*{\gamma_j};
\endxy
\qquad
\qquad
\xy /l2pc/:
(0,-2)*{\bullet}="21";
(3,0)*{\bullet}="12";
(1,0)*{\times}="22";
(-1,0)*{\times}="32";
(-3,0)*{\bullet}="42";
(0,2)*{\bullet}="23";
(0,0)*{}="0";
{"22"\PATH~={**@{~}}'"32"},
{"12"\PATH~={**@{-}}'"22"},
{"32"\PATH~={**@{-}}'"42"},
{"21"\PATH~={**@{-}}'"22"},
{"21"\PATH~={**@{-}}'"32"},
{"22"\PATH~={**@{-}}'"23"},
{"32"\PATH~={**@{-}}'"23"},
(0,-2.4)*{\ominus};
(0,2.4)*{\ominus};
(-3.4,0)*{\oplus};
(3.4,0)*{\oplus};
{"21"\PATH~={**@{-}}'"0"},
{"0"\PATH~={**@{.}}'"23"},
(0,-1)*{\vee};
(0,1)*{\vee};
(-0.25,-0.25)*{\beta_j};
\endxy
\]
In these pictures, we illustrate the portion of the first sheet that lies over $D_j$. In the second picture, the dotted line indicates the portion of the curve that lies on the second sheet of the cover. On the other hand, suppose $j$ and $k$ are generic trajectories contained in the horizontal strips~$D_j$ and~$D_k$, respectively, where $D_j$ is a degenerate horizontal strip, and $D_k$ is the regular horizontal strip surrounding~$D_j$. Then we define $\gamma_j$,~$\gamma_k\in\mathcal{H}$ as in the pictures 
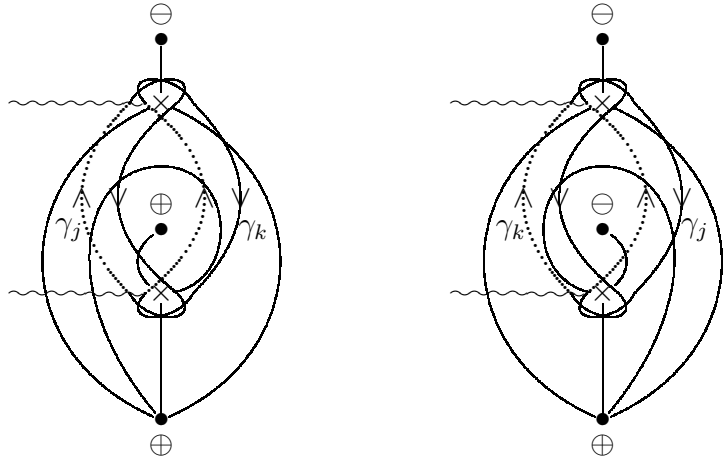
\[
\xy /l2pc/:
(0,-3)*{\bullet}="1";
(0,-2)*{\times}="2";
(0,0)*{\bullet}="3";
(0,1)*{\times}="4";
(0,3)*{\bullet}="5";
(2.5,-2)*{}="a";
(0,-1)*{}="b";
(2.5,1)*{}="c";
{"1"\PATH~={**@{-}}'"2"},
{"4"\PATH~={**@{-}}'"5"},
{"a"\PATH~={**@{~}}'"2"},
{"c"\PATH~={**@{~}}'"4"},
"2";"5" **\crv{(2.5,-1) & (2.5,2)}; 
"2";"5" **\crv{(-2.5,-1) & (-2.5,2)}; 
"3";"4" **\crv{(0.5,0.25) & (0.5,0.75)}; 
"b";"5" **\crv{(1.25,-1) & (1.75,1)}; 
"b";"4" **\crv{(-1.25,-1) & (-1.25,1)}; 
(0,-0.4)*{\oplus};
(0,-3.4)*{\ominus};
(0,3.4)*{\oplus};
(0.5,-2)*{}="x";
(-0.25,-2)*{}="y";
(0.5,1)*{}="z";
(-0.25,1)*{}="w";
(-0.5,-2)*{}="q";
(0.25,-2)*{}="r";
(-0.5,1)*{}="s";
(0.25,1)*{}="t";
"x";"y" **\crv{(0.25,-2.5) & (-0.7,-2.5)}; 
"y";"w" **\crv{(1,-1) & (1,0)}; 
"z";"w" **\crv{(0.25,1.5) & (-0.7,1.5)}; 
"x";"z" **\crv{~*=<2pt>{.} (1.5,-1) & (1.5,0)}; 
"q";"r" **\crv{(-0.25,-2.5) & (0.7,-2.5)}; 
"r";"t" **\crv{~*=<2pt>{.} (-1,-1) & (-1,0)}; 
"s";"t" **\crv{(-0.25,1.5) & (0.7,1.5)}; 
"q";"s" **\crv{(-1.5,-1) & (-1.5,0)}; 
(1.25,-0.5)*{\wedge};
(0.68,-0.5)*{\vee};
(-0.68,-0.5)*{\wedge};
(-1.25,-0.5)*{\vee};
(1.45,0)*{\gamma_j};
(-1.45,0)*{\gamma_k};
\endxy
\qquad
\qquad
\xy /l2pc/:
(0,-3)*{\bullet}="1";
(0,-2)*{\times}="2";
(0,0)*{\bullet}="3";
(0,1)*{\times}="4";
(0,3)*{\bullet}="5";
(2.5,-2)*{}="a";
(0,-1)*{}="b";
(2.5,1)*{}="c";
{"1"\PATH~={**@{-}}'"2"},
{"4"\PATH~={**@{-}}'"5"},
{"a"\PATH~={**@{~}}'"2"},
{"c"\PATH~={**@{~}}'"4"},
"2";"5" **\crv{(2.5,-1) & (2.5,2)}; 
"2";"5" **\crv{(-2.5,-1) & (-2.5,2)}; 
"3";"4" **\crv{(-0.5,0.25) & (-0.5,0.75)}; 
"b";"5" **\crv{(-1.25,-1) & (-1.75,1)}; 
"b";"4" **\crv{(1.25,-1) & (1.25,1)}; 
(0,-0.4)*{\ominus};
(0,-3.4)*{\ominus};
(0,3.4)*{\oplus};
(0.5,-2)*{}="x";
(-0.25,-2)*{}="y";
(0.5,1)*{}="z";
(-0.25,1)*{}="w";
(-0.5,-2)*{}="q";
(0.25,-2)*{}="r";
(-0.5,1)*{}="s";
(0.25,1)*{}="t";
"x";"y" **\crv{(0.25,-2.5) & (-0.7,-2.5)}; 
"y";"w" **\crv{(1,-1) & (1,0)}; 
"z";"w" **\crv{(0.25,1.5) & (-0.7,1.5)}; 
"x";"z" **\crv{~*=<2pt>{.} (1.5,-1) & (1.5,0)}; 
"q";"r" **\crv{(-0.25,-2.5) & (0.7,-2.5)}; 
"r";"t" **\crv{~*=<2pt>{.} (-1,-1) & (-1,0)}; 
"s";"t" **\crv{(-0.25,1.5) & (0.7,1.5)}; 
"q";"s" **\crv{(-1.5,-1) & (-1.5,0)}; 
(1.25,-0.5)*{\wedge};
(0.68,-0.5)*{\vee};
(-0.68,-0.5)*{\wedge};
(-1.25,-0.5)*{\vee};
(1.45,0)*{\gamma_k};
(-1.45,0)*{\gamma_j};
\endxy
\]
and we define $\beta_j$,~$\beta_k\in\mathcal{H}^\vee$ as in the pictures
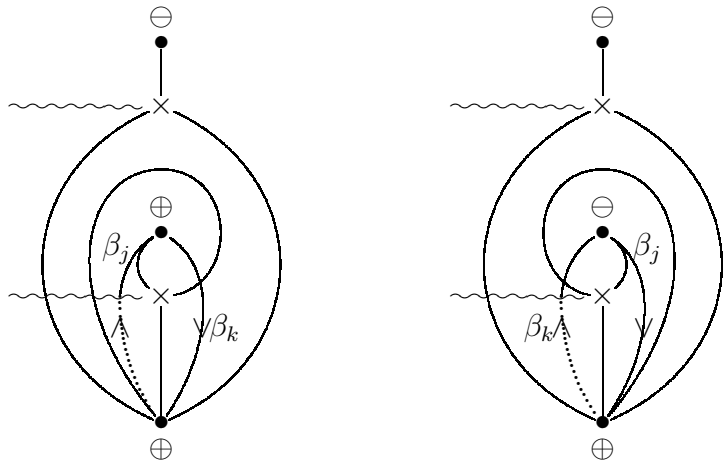
\[
\xy /l2pc/:
(0,-3)*{\bullet}="1";
(0,-2)*{\times}="2";
(0,0)*{\bullet}="3";
(0,1)*{\times}="4";
(0,3)*{\bullet}="5";
(2.5,-2)*{}="a";
(0,-1)*{}="b";
(2.5,1)*{}="c";
{"1"\PATH~={**@{-}}'"2"},
{"4"\PATH~={**@{-}}'"5"},
{"a"\PATH~={**@{~}}'"2"},
{"c"\PATH~={**@{~}}'"4"},
"2";"5" **\crv{(2.5,-1) & (2.5,2)}; 
"2";"5" **\crv{(-2.5,-1) & (-2.5,2)}; 
"3";"4" **\crv{(0.5,0.25) & (0.5,0.75)}; 
"b";"5" **\crv{(1.25,-1) & (1.75,1)}; 
"b";"4" **\crv{(-1.25,-1) & (-1.25,1)}; 
(0,-0.4)*{\oplus};
(0,-3.4)*{\ominus};
(0,3.4)*{\oplus};
(0.65,1)*{}="m";
"3";"5" **\crv{(-1,0.5) & (-0.75,2)}; 
"3";"m" **\crv{(0.5,0.2) & (0.64,0.78)}; 
"m";"5" **\crv{~*=<2pt>{.} (0.65,0.8) & (0.75,2.1)};
(-0.65,1.5)*{\vee};
(0.65,1.5)*{\wedge};
(0.7,0.25)*{\beta_j};
(-1,1.5)*{\beta_k};
\endxy
\qquad
\qquad
\xy /l2pc/:
(0,-3)*{\bullet}="1";
(0,-2)*{\times}="2";
(0,0)*{\bullet}="3";
(0,1)*{\times}="4";
(0,3)*{\bullet}="5";
(2.5,-2)*{}="a";
(0,-1)*{}="b";
(2.5,1)*{}="c";
{"1"\PATH~={**@{-}}'"2"},
{"4"\PATH~={**@{-}}'"5"},
{"a"\PATH~={**@{~}}'"2"},
{"c"\PATH~={**@{~}}'"4"},
"2";"5" **\crv{(2.5,-1) & (2.5,2)}; 
"2";"5" **\crv{(-2.5,-1) & (-2.5,2)}; 
"3";"4" **\crv{(-0.5,0.25) & (-0.5,0.75)}; 
"b";"5" **\crv{(-1.25,-1) & (-1.75,1)}; 
"b";"4" **\crv{(1.25,-1) & (1.25,1)}; 
(0,-0.4)*{\ominus};
(0,-3.4)*{\ominus};
(0,3.4)*{\oplus};
(0.65,1)*{}="m";
"3";"5" **\crv{(-1,0.5) & (-0.75,2)}; 
"3";"m" **\crv{(0.5,0.2) & (0.64,0.78)}; 
"m";"5" **\crv{~*=<2pt>{.} (0.65,0.8) & (0.75,2.1)};
(-0.65,1.5)*{\vee};
(0.65,1.5)*{\wedge};
(-0.7,0.25)*{\beta_j};
(1,1.5)*{\beta_k};
\endxy
\]

\section{Borel resummation}
\label{sec:BorelResummation}

\subsection{Definition and basic properties}

In Section~\ref{sec:BackgroundOnWKBAnalysis}, we constructed the WKB solutions, a pair of formal solutions of the Schr\"odinger equation. As we have already mentioned, these formal series are generally divergent, and therefore, in order to get genuine analytic solutions, we employ a renormalization method known as the Borel resummation. In this short section, we review those aspects of the Borel resummation method that will be important in what follows. Our treatment is based on~\cite{IwakiNakanishi}; further details can be found in~\cite{Costin}.

\Needspace*{4\baselineskip}
\begin{definition}[\cite{IwakiNakanishi}, Definition~2.10] \mbox{}
\begin{enumerate}
\item A formal power series $f(\hbar)=\sum_{n=0}^\infty f_n\hbar^n$ is said to be \emph{Borel summable} if the formal power series 
\[
f_B(y)=\sum_{n=1}^\infty f_n\frac{y^{n-1}}{(n-1)!}
\]
converges near $y=0$, possesses an analytic continuation to a domain $\Omega$ containing the half line $\{y\in\mathbb{C}:\Re(y)\geq0,\Im(y)=0\}$, and satisfies the bound 
\[
|f_B(y)|\leq C_1e^{C_2|y|}
\]
on~$\Omega$ for some constants $C_1$,~$C_2>0$.
\item If $f(\hbar)=\sum_{n=0}^\infty f_n\hbar^n$ is a Borel summable formal power series, we define the \emph{Borel sum} of $f(\hbar)$ by 
\[
\mathcal{S}[f](\hbar)=f_0+\int_0^\infty e^{-y/\hbar}f_B(y)dy
\]
where the integral is taken along the positive real axis.
\item An expression of the form $f(\hbar)=e^{s/\hbar}\hbar^\rho\sum_{n=0}^\infty f_n\hbar^n$, where $g(\hbar)=\sum_{n=0}^\infty f_n\hbar^n$ is a formal power series and $\rho$,~$s\in\mathbb{C}$, is said to be \emph{Borel summable} if $g(\hbar)$ is Borel summable. In this case, its \emph{Borel sum} is defined as $\mathcal{S}[f](\hbar)=e^{s/\hbar}\hbar^\rho\mathcal{S}[g](\hbar)$.
\end{enumerate}
\end{definition}

We will use the following properties of Borel sums throughout our discussion.

\Needspace*{2\baselineskip}
\begin{proposition}[\cite{IwakiNakanishi}, Proposition~2.11]\mbox{}
\begin{enumerate}
\item If $f(\hbar)$ is a Borel summable formal power series, then the Borel sum $\mathcal{S}[f](\hbar)$ is an analytic function on a domain 
\[
\mathbb{H}(\varepsilon)=\{\hbar\in\mathbb{C}:|\hbar|<\epsilon \text{ and } \Re(\hbar)>0\}
\]
for some $\varepsilon>0$. Moreover, $\mathcal{S}[f](\hbar)$ is asymptotically expanded to $f(\hbar)$ when $\hbar\rightarrow0$.
\item If $f(\hbar)$ and $g(\hbar)$ are Borel summable, then 
\[
\mathcal{S}[f+g]=\mathcal{S}[f]+\mathcal{S}[g], \quad \mathcal{S}[f\cdot g]=\mathcal{S}[f]\cdot\mathcal{S}[g].
\]
\item Let $A(t)=\sum_{k=0}^\infty A_kt^k$ be a power series that converges near $t=0$. If a formal power series $f(\hbar)=\sum_{n=1}^\infty f_n\hbar^n$ is Borel summable, then the formal power series $A(f(\hbar))=\sum_{k=0}^\infty A_k(f(\hbar))^k$ is also Borel summable, and its Borel sum is given by $\mathcal{S}[A(f(\hbar))]=A(\mathcal{S}[f](\hbar))$.
\end{enumerate}
\end{proposition}

\subsection{Borel sums of WKB solutions and Voros symbols}

In Section~\ref{sec:VorosSymbolsAsClusterCoordinates}, we will study the Borel sums of the WKB solutions and Voros symbols, and we will prove that the latter coincide with Fock-Goncharov coordinates of framed local systems. Therefore the Borel summability of these formal series is a crucial result. As explained in~\cite{IwakiNakanishi}, this follows from unpublished work of Koike and~Sch\"afke which establishes the Borel summability of $S_\text{odd}(z,\hbar)$. For a sketch of the proof of this result in a more specialized context, see~\cite{Takei17}.

\begin{theorem}[\cite{IwakiNakanishi}, Corollary~2.21]
\label{thm:Borelsummability}
Let $\phi$ be a complete, saddle-free GMN differential satisfying Assumptions~\ref{assumption2.5}.
\begin{enumerate}
\item Let $\beta$ be a path in the spectral cover that projects to a generic trajectory of~$\phi$. Then the formal power series $\int_\beta S_{\text{odd}}^{\text{reg}}(z,\hbar)dz$ is Borel summable.
\item Let $D$ be a horizontal strip or half plane defined by~$\phi$. Then a WKB solution, normalized as in Definition~\ref{def:normalizedWKB}, is Borel summable at each point in~$D$. The Borel sum is an analytic solution of~\eqref{eqn:schrodinger} on~$D$, which is also analytic in $\hbar$ in a domain $\mathbb{H}(\varepsilon)$ for some $\varepsilon>0$.
\end{enumerate}
\end{theorem}

In fact, one can relax the assumptions of this theorem so that $\phi$ has \emph{at most one saddle connection}, but we will not need this more general result.

\section{Projective structures}
\label{sec:ProjectiveStructures}

\subsection{Classical theory of projective structures}

In the next section, we will review the monodromy map from~\cite{AllegrettiBridgeland1}. In this section, we set the stage by reviewing some facts about projective structures on a Riemann surface.

\begin{definition}
Let $S$ be a Riemann surface. A \emph{projective structure} on~$S$ is a maximal atlas of holomorphic charts $z_\alpha:U_\alpha\rightarrow\mathbb{CP}^1$, $S=\bigcup_{\alpha}U_\alpha$, such that each transition function $z_\alpha\circ z_\beta^{-1}$ is the restriction of a M\"obius transformation.
\end{definition}

Note that by the uniformization theorem, any connected Riemann surface $S$ can be written as $S=\tilde{S}/\Gamma$ where the universal cover $\tilde{S}$ is either the Riemann sphere, the complex plane, or the upper half plane, and $\Gamma\subseteq PGL_2(\mathbb{C})$ is a discrete subgroup. By taking local sections of the covering map, we see that $S$ has a projective structure.

In fact, we can typically get many other projective structures on~$S$ by adding a holomorphic quadratic differential.

\begin{proposition}
The set of projective structures on a Riemann surface~$S$ is an affine space modeled on the vector space $H^0(S,\omega_S^{\otimes2})$ of holomorphic quadratic differentials.
\end{proposition}

Let us sketch the idea behind this statement without giving the full details. Let $\mathcal{P}$ be a projective structure on~$S$, and let $\phi$ be a holomorphic quadratic differential on~$S$. If $z:U\rightarrow\mathbb{CP}^1$ is a chart belonging to the projective structure $\mathcal{P}$, then we can express $\phi$ in this local coordinate as 
\[
\phi(z)=\varphi(z)dz^{\otimes2}
\]
for some meromorphic function $\varphi(z)$. Let $y_1(z)$ and $y_2(z)$ be linearly independent solutions of the differential equation 
\begin{align}
\label{eqn:schrodingerwithoutparameter}
y''(z)-\varphi(z)y(z)=0.
\end{align}
Then the ratio $w\coloneqq y_1(z)/y_2(z)$ defines a chart in a new projective structure. We denote this projective structure by $\mathcal{P}+\phi$.

On the other hand, if $\mathcal{P}_1$ and $\mathcal{P}_2$ are two projective structures on~$S$, let $z_1:U\rightarrow\mathbb{CP}^1$ and $z_2:U\rightarrow\mathbb{CP}^1$ be charts in~$\mathcal{P}_1$ and~$\mathcal{P}_2$, respectively. Then we can define a quadratic differential on the set~$U$ by 
\[
-\frac{1}{2}\{z_2(z_1),z_1\}dz_1^{\otimes2}
\]
where $\{z_2(z_1),z_1\}$ is the Schwarzian derivative defined in Section~\ref{sec:BackgroundOnWKBAnalysis}. These local expressions give rise to a global holomorphic quadratic differential $\phi$ on~$S$. We denote this quadratic differential by $\phi=\mathcal{P}_2-\mathcal{P}_1$. One can show that these constructions satisfy Weyl's axioms, making the set of projective structures into an affine space.

An important fact in the theory of projective structures is the existence of developing maps.

\begin{lemma}[\cite{Hubbard}, Lemma~1]
\label{lem:developing}
Let $\mathcal{P}$ be a projective structure on a Riemann surface $S$, and let $p:\tilde{S}\rightarrow S$ be the universal cover. Then there exists a holomorphic map $f:\tilde{S}\rightarrow\mathbb{CP}^1$ such that if $U\subseteq S$ is a contractible open subset, then the composition $f\circ p^{-1}|_U$ is a chart in~$\mathcal{P}$.
\end{lemma}

A map $f$ as in Lemma~\ref{lem:developing} is called a \emph{developing map} and is essentially unique: Any two developing maps differ by post-composition with a M\"obius transformation. Note that if $\gamma\in\pi_1(S)$, then $f\circ\gamma$ is again a developing map, and so there exists $\rho(\gamma)\in PGL_2(\mathbb{C})$ such that 
\[
f\circ\gamma=\rho(\gamma)\circ f.
\]
In this way, we obtain a map $\rho$ from the fundamental group into $PGL_2(\mathbb{C})$, well defined up to conjugation.

\begin{definition}
The map $\rho:\pi_1(S)\rightarrow PGL_2(\mathbb{C})$ is called the \emph{monodromy} of the projective structure.
\end{definition}

\subsection{Meromorphic projective structures}

We will now consider projective structures with poles of prescribed orders. This notion has meaning because of the fact reviewed above that the set of projective structures is an affine space for the vector space of holomorphic quadratic differentials.

\begin{definition}[\cite{AllegrettiBridgeland1}, Definition~3.1]
\label{def:meromorphicprojectivestructure}
A \emph{meromorphic projective structure} $\mathcal{P}$ is defined to be a projective structure $\mathcal{P}^*$ on the complement $S^*=S\setminus D$ of a discrete subset $D\subset S$ such that for each ordinary projective structure $\mathcal{P}_0$ on~$S$, the difference 
\[
\phi=\mathcal{P}^*-\mathcal{P}_0|_{S^*}
\]
extends to a meromorphic quadratic differential on~$S$.
\end{definition}

It is easy to see, using the relationship between projective structures and quadratic differentials, that the meromorphic quadratic differential $\phi$ appearing in the above definition is well defined up to addition of holomorphic quadratic differentials.

\begin{definition}
The meromorphic quadratic differential $\phi$ appearing in Definition~\ref{def:meromorphicprojectivestructure} is called the \emph{polar differential} of the meromorphic projective structure~$\mathcal{P}$. By a \emph{pole} of~$\mathcal{P}$, we mean a pole of the polar differential.
\end{definition}

Let $\mathcal{P}$ be a meromorphic projective structure on a Riemann surface $S$, and let $D$ denote the set of poles of~$\mathcal{P}$. By definition, $\mathcal{P}$ induces an ordinary projective structure on the punctured surface $S^*=S\setminus D$. By applying the results that we reviewed above, we therefore obtain a monodromy representation 
\[
\rho:\pi_1(S^*)\rightarrow PGL_2(\mathbb{C}).
\]

When we construct moduli spaces below, it will be important to eliminate the degenerate situation described in the following definition.

\begin{definition}
Let $\mathcal{P}$ be a meromorphic projective structure on a Riemann surface $S$, and let $p\in S$ be a pole of~$\mathcal{P}$. We say that $p$ is an \emph{apparent singularity} of~$\mathcal{P}$ if the monodromy of the class of a small loop around~$p$ is the identity in~$PGL_2(\mathbb{C})$.
\end{definition}

\subsection{Moduli space of marked projective structures}

Let $(\mathbb{S},\mathbb{M})$ be a fixed marked bordered surface. Consider a meromorphic projective structure $\mathcal{P}$ on a compact Riemann surface~$S$. Because the polar differential is well defined up to addition of holomorphic quadratic differentials, we see that there is a marked bordered surface canonically associated to~$\mathcal{P}$. It is defined as the marked bordered surface associated to the polar differential.

\begin{definition}
By a \emph{marking} of the pair $(S,\mathcal{P})$ by $(\mathbb{S},\mathbb{M})$ we mean an isotopy class of isomorphisms between $(\mathbb{S},\mathbb{M})$ and the marked bordered surface associated to~$\mathcal{P}$. A \emph{marked projective structure} is defined to be a triple of the form $(S,\mathcal{P},\theta)$ where $S$ is a compact Riemann surface equipped with a meromorphic projective structure~$\mathcal{P}$ and $\theta$ is a marking of the pair $(S,\mathcal{P})$ by~$(\mathbb{S},\mathbb{M})$.
\end{definition}

Two marked projective structures $(S_1,\mathcal{P}_1,\theta_1)$ and $(S_2,\mathcal{P}_2,\theta_2)$ are considered to be equivalent if there is a biholomorphism $S_1\rightarrow S_2$ between the underlying Riemann surfaces which preserves the projective structures $\mathcal{P}_i$ and commutes with the markings $\theta_i$ in the obvious way.

\begin{proposition}[\cite{AllegrettiBridgeland1}, Proposition~8.2]
Let $(\mathbb{S},\mathbb{M})$ be a marked bordered surface, and if $\mathbb{S}$ has genus~0, assume that $|\mathbb{M}|\geq3$. Then the set $\Proj(\mathbb{S},\mathbb{M})$ of equivalence classes of marked projective structures has the structure of a complex manifold.
\end{proposition}

In fact, it is convenient to modify this space of projective structures in two ways. Firstly, we restrict attention to the open subset 
\[
\Proj^\circ(\mathbb{S},\mathbb{M})\subseteq\Proj(\mathbb{S},\mathbb{M})
\]
whose complement is the codimension two locus of projective structures with apparent singularities. We do this because the monodromy map of~\cite{AllegrettiBridgeland1} is not defined at projective structures with apparent singularities. Secondly, we consider a finite cover 
\[
\Proj^*(\mathbb{S},\mathbb{M})\rightarrow\Proj^\circ(\mathbb{S},\mathbb{M})
\]
where a point in the fiber over a marked projective structure $(S,\mathcal{P},\theta)$ consists of a choice of eigenline for the monodromy around each pole of order $\leq2$. We do this so that the monodromy map takes values in the stack of framed $PGL_2(\mathbb{C})$-local systems which is rational and carries an interesting atlas of rational cluster coordinates as we explain below.

\begin{proposition}[\cite{AllegrettiBridgeland1}, Proposition~8.4]
Let $(\mathbb{S},\mathbb{M})$ be a marked bordered surface, and if $\mathbb{S}$ has genus~0, assume that $|\mathbb{M}|\geq3$. Then the set $\Proj^*(\mathbb{S},\mathbb{M})$ has the structure of a complex manifold.
\end{proposition}

\section{Framed local systems}
\label{sec:FramedLocalSystems}

\subsection{Moduli space of framed local systems}

Let $(\mathbb{S},\mathbb{M})$ be a marked bordered surface, and consider the punctured surface $\mathbb{S}^*=\mathbb{S}\setminus\mathbb{P}$. Let $\mathcal{L}$ be a $PGL_2(\mathbb{C})$-local system on~$\mathbb{S}^*$, that is, a principal $PGL_2(\mathbb{C})$-bundle with flat connection. There is a natural left action of $PGL_2(\mathbb{C})$ on the projective line, and hence we can form the associated bundle 
\[
\mathcal{L}_{\mathbb{CP}^1}\coloneqq\mathcal{L}\times_{PGL_2(\mathbb{C})}\mathbb{CP}^1.
\]
For each marked point $p\in\mathbb{M}$, choose a small neighborhood $U(p)\subseteq\mathbb{S}$ of~$p$.

\begin{definition}
A \emph{framing} for a $PGL_2(\mathbb{C})$-local system on $\mathbb{S}^*$ is defined as a flat section of the restriction of $\mathcal{L}_{\mathbb{CP}^1}$ to $U(p)\cap\mathbb{S}^*$ for each $p\in\mathbb{M}$. A \emph{framed $PGL_2(\mathbb{C})$-local system} on~$(\mathbb{S},\mathbb{M})$ is a $PGL_2(\mathbb{C})$-local system on the punctured surface~$\mathbb{S}^*$ together with a framing.
\end{definition}

Two framed local systems are isomorphic if there is an isomorphism of the underlying local systems such that the induced map on associated $\mathbb{CP}^1$-bundles preserves the framings.

Fix a basepoint $x\in\mathbb{S}^*$. By a \emph{rigidified framed local system} on~$(\mathbb{S},\mathbb{M})$, we mean a pair $(\mathcal{L},s)$ where $\mathcal{L}$ is a framed local system and $s$ is an element of the fiber of $\mathcal{L}_{\mathbb{CP}^1}$ over~$x$. An isomorphism between rigidified framed local systems is an isomorphism of the underlying framed local systems that preserves the chosen points in the fibers of the associated $\mathbb{CP}^1$-bundles. To understand this definition concretely, let us choose, for each point $p\in\mathbb{M}$, a path $\beta_p$ connecting $x$ to~$p$ whose interior lies in~$\mathbb{S}^*$. Then, for each puncture $p$, we can define a loop $\delta_p\in\pi_1(S,x)$ which travels from $x$ to $U(p)$ along $\beta_p$, travels along a small loop around~$p$ in the counterclockwise direction, and then returns to $x$ along~$\beta_p$.

\begin{lemma}[\cite{AllegrettiBridgeland1}, Lemma~4.2]
There is a bijection between the set of isomorphism classes of rigidified framed $PGL_2(\mathbb{C})$-local systems on~$(\mathbb{S},\mathbb{M})$ and the set of points of the complex projective variety 
\begin{align*}
X(\mathbb{S},\mathbb{M})=\big\{(\rho,\phi):
& \, \rho\in\Hom_{\mathrm{Grp}}\left(\pi_1(\mathbb{S}^*,x),PGL_2(\mathbb{C})\right),
\phi\in\Hom_{\mathrm{Set}}(\mathbb{M},\mathbb{CP}^1), \\
& \quad\rho(\delta_p)(\phi(p))=\phi(p) \text{ for $p\in\mathbb{P}$}
\big\}.
\end{align*}
\end{lemma}

Under this bijection, the homomorphism $\rho$ is the monodromy representation of a rigidified framed local system while the map $\phi$ is defined by parallel transporting the framing lines to the basepoint $x$ along the paths $\beta_p$. An element $g\in PGL_2(\mathbb{C})$ acts naturally on a rigidified framed local system $(\mathcal{L},s)$ by fixing the underlying framed local system $\mathcal{L}$ and mapping $s\mapsto gs$. The induced action on the variety $X(\mathbb{S},\mathbb{M})$ is given by 
\[
g\cdot(\rho,\phi)=(g\rho g^{-1},g\phi).
\]
The moduli stack of framed $PGL_2(\mathbb{C})$-local systems on $(\mathbb{S},\mathbb{M})$ is then defined as the quotient 
\[
\mathcal{X}(\mathbb{S},\mathbb{M})=X(\mathbb{S},\mathbb{M})/PGL_2(\mathbb{C}).
\]

\subsection{Fock-Goncharov coordinates}

In~\cite{FG1}, Fock and Goncharov defined an atlas of rational coordinate charts on the moduli stack of framed $PGL_2(\mathbb{C})$-local systems. These coordinate charts are indexed by ideal triangulations of $(\mathbb{S},\mathbb{M})$. Here we will describe a slight extension of their construction in which the coordinate charts are indexed by combinatorial objects called tagged triangulations~\cite{FST}.

Let us equip the surface $\mathbb{S}'=\mathbb{S}\setminus\mathbb{M}$ with a complete, finite area hyperbolic metric with totally geodesic boundary. Then the universal cover of $\mathbb{S}'$ can be identified with a subset of the hyperbolic plane $\mathbb{H}$ with totally geodesic boundary. The deleted marked points of $\mathbb{S}'$ give rise to a set of points on the boundary $\partial\mathbb{H}$. This set is known as the \emph{Farey set} and is denoted $\mathcal{F}_\infty(\mathbb{S},\mathbb{M})$. Any ideal triangulation of~$(\mathbb{S},\mathbb{M})$ can be lifted to a triangulation of the universal cover whose vertices are the points of $\mathcal{F}_\infty(\mathbb{S},\mathbb{M})$. Moreover, a framed local system on $(\mathbb{S},\mathbb{M})$ naturally determines a map $\psi:\mathcal{F}_\infty(\mathbb{S},\mathbb{M})\rightarrow\mathbb{CP}^1$.

\begin{definition}
Choose a general point of $X(\mathbb{S},\mathbb{M})$, and let $T$ be an ideal triangulation of~$(\mathbb{S},\mathbb{M})$. To any arc $j$ of $T$, we associate a number $X_j\in\mathbb{C}^*$ as follows:
\begin{enumerate}
\item Let $\tilde{j}$ be a lift of $j$ to the universal cover. Then there are two triangles of the lifted triangulation that share the side $\tilde{j}$, and these form a quadrilateral in~$\mathbb{H}$. Let $c_1$, $c_2$, $c_3$, and~$c_4$ be the vertices of this quadrilateral in counterclockwise order so that the arc $\tilde{j}$ joins the vertices $c_1$ and $c_3$. For each index $i$, let $v_i$ be a nonzero vector in the line $\psi(c_i)$. Then we define 
\[
Y_j=\frac{\det(v_1v_2)\det(v_3v_4)}{\det(v_2v_3)\det(v_1v_4)}
\]
where $\det(v_sv_t)$ denotes the determinant of the $2\times2$ matrix having columns $v_s$ and~$v_t$ in this order. We can assume these determinants are nonzero since we work with a general point of $X(\mathbb{S},\mathbb{M})$. Note that there are two ways of ordering the points $c_i$ and they give the same value for the cross ratio.
\item If $j$ is not the interior edge of a self-folded triangle, then we define $X_j=Y_j$. If $j$ is the interior edge of a self-folded triangle, let $k$ be the encircling loop. In this case, we define $X_j=Y_jY_k$.
\end{enumerate}
\end{definition}

Thus we associate to a general point of $X(\mathbb{S},\mathbb{M})$ a tuple of numbers $X_j\in\mathbb{C}^*$ indexed by the arcs of the ideal triangulation~$T$. These numbers are invariant under the action of the group $PGL_2(\mathbb{C})$ on $X(\mathbb{S},\mathbb{M})$, and so we get a rational map 
\[
X_T:\mathcal{X}(\mathbb{S},\mathbb{M})\dashrightarrow(\mathbb{C}^*)^n
\]
where $n$ is the number of arcs in~$T$.

\begin{lemma}[\cite{AllegrettiBridgeland1}, Lemma~9.3]
For any ideal triangulation $T$, the map $X_T$ is a birational equivalence.
\end{lemma}

We can extend this result by equipping the ideal triangulation $T$ with an extra datum. Namely, we define a \emph{signing} for $T$ to be a function 
\[
\epsilon:\mathbb{P}\rightarrow\{\pm1\},
\]
and we define a \emph{signed triangulation} to be an ideal triangulation equipped with a signing. Two signed triangulations $(T_1,\epsilon_1)$ and $(T_2,\epsilon_2)$ are considered to be equivalent if $T_1=T_2$ and the signings differ at the interior puncture of a self-folded triangle. An equivalence class of signed triangulations is called a \emph{tagged triangulation}. Note that any ideal triangulation may be considered as a tagged triangulation by taking the signing $\epsilon\equiv+1$.

Suppose $\tau$ is a tagged triangulation represented by a signed triangulation $(T,\epsilon)$. By a \emph{tagged arc} of $\tau$, we mean an arc of $T$. If $(T,\epsilon')$ is another signed triangulation where $\epsilon'$ differs from $\epsilon$ only at the interior puncture of a self-folded triangle, let $j$ be the interior edge of this self-folded triangle, and let $k$ be the encircling loop. Then the tagged arc represented by $j$ in $(T,\epsilon)$ is considered to be equivalent to the tagged arc represented by $k$ in $(T,\epsilon')$.

There is an obvious action of the group $(\mathbb{Z}/2\mathbb{Z})^{\mathbb{P}}$ on the set of signed or tagged triangulations. Note that, for a generic framed local system on $(\mathbb{S},\mathbb{M})$, the monodromy around a puncture $p$ is semisimple. The framing at $p$ is given by an eigenline of the monodromy, and hence for a generic framed local system, there is a unique way of modifying the framing at~$p$ to get a different framed local system.

\begin{lemma}[\cite{AllegrettiBridgeland1}, Lemma~9.4]
There is a birational action of the group $(\mathbb{Z}/2\mathbb{Z})^{\mathbb{P}}$ on the stack $\mathcal{X}(\mathbb{S},\mathbb{M})$ of framed local systems in which the nontrivial generator corresponding to $p\in\mathbb{P}$ acts by fixing the underlying local system and exchanging the two generically possible choices of framing at~$p$.
\end{lemma}

Using this birational action, we can extend the construction of the Fock-Goncharov coordinates.

\begin{definition}[\cite{AllegrettiBridgeland1}, Definition~9.5]
The \emph{Fock-Goncharov coordinate} of a framed local system $\mu$ with respect to an arc $j$ of the signed triangulation $(T,\epsilon)$ is defined to be the coordinate $X_j$ of the framed local system obtained by applying the group element $\epsilon\in(\mathbb{Z}/2\mathbb{Z})^{\mathbb{P}}$ to $\mu$, whenever this quantity is well defined.
\end{definition}

One can show that the Fock-Goncharov coordinate associated to an arc of a signed triangulation depends only on the underlying tagged arc. Thus we get a birational equivalence 
\[
X_\tau:\mathcal{X}(\mathbb{S},\mathbb{M})\dashrightarrow(\mathbb{C}^*)^n
\]
for each tagged triangulation $\tau$. The transition maps for transforming between different coordinate charts are given by cluster transformations, and hence the stack of framed local systems is birational to a cluster variety. For details, see~\cite{AllegrettiBridgeland1}.

\subsection{Generalized monodromy map}

We now recall the definition of the generalized monodromy map from~\cite{AllegrettiBridgeland1}. This is a holomorphic map 
\[
F:\Proj^*(\mathbb{S},\mathbb{M})\rightarrow\mathcal{X}(\mathbb{S},\mathbb{M}).
\]
As we have explained, a point in the space on the left consists of a meromorphic projective structure $\mathcal{P}$ on a compact Riemann surface~$S$, together with a marking $\theta$ of $(S,\mathcal{P})$ by $(\mathbb{S},\mathbb{M})$ and a choice collection of invariant lines $\{\ell_p\}_{p\in\mathbb{P}}$. The projective structure $\mathcal{P}$ determines a monodromy representation, which can be viewed as a homomorphism $\pi_1(\mathbb{S}^*)\rightarrow PGL_2(\mathbb{C})$. Thus we obtain a $PGL_2(\mathbb{C})$-local system on $\mathbb{S}^*$.

The main point in the construction of the map $F$ is that this local system is naturally equipped with a framing. Near a pole $p$ of order $\leq2$, this framing is provided by the line $\ell_p$. On the other hand, suppose $p\in S$ is a pole of $\mathcal{P}$ of order $m\geq3$. If $\phi$ is a polar differential for $\mathcal{P}$, then we can choose a local coordinate $z$ such that $z(p)=0$ and the differential is given in this local coordinate by $\phi(z)=\varphi(z)dz^{\otimes2}$ where 
\[
\varphi(z)=z^{-m}(a_0+a_1z+a_2z^2+\dots).
\]
We define the \emph{asymptotic horizontal directions} of~$\phi$ at~$p$ to be the tangent directions determined by the condition $a_0z^{2-m}\in\mathbb{R}_{>0}$. These coincide with the distinguished tangent directions determined by the horizontal foliation of~$\phi$. We define the \emph{asymptotic vertical directions} by the condition $a_0z^{2-m}\in\mathbb{R}_{<0}$.

\begin{proposition}[\cite{AllegrettiBridgeland1}, Theorem~5.2]
\label{prop:subdominantexists}
Let $\mathscr{S}$ be a sector bounded by two adjacent asymptotic vertical directions. Then there is a unique-up-to-scale holomorphic solution $y(z)$ of~\eqref{eqn:schrodingerwithoutparameter} such that $y(z)\rightarrow0$ as $z\rightarrow0$ in~$\mathscr{S}$.
\end{proposition}

A solution $y(z)$ as in Proposition~\ref{prop:subdominantexists} is said to be \emph{subdominant} in the sector~$\mathscr{S}$. It follows that for each marked point on the boundary of the marked bordered surface $(\mathbb{S},\mathbb{M})$ associated to the differential $\phi$, there is a distinguished one-dimensional subspace of the space of all solutions of~\eqref{eqn:schrodingerwithoutparameter} spanned by the subdominant solutions. This defines the framing of our local system near marked points on the boundary of $(\mathbb{S},\mathbb{M})$.

Thus we have constructed the map $F$. In~\cite{AllegrettiBridgeland1}, we showed that any point in the image of this map is a regular point of $X_\tau$ for some tagged triangulation~$\tau$. Since the $X_\tau$ map into the algebraic torus $(\mathbb{C}^*)^n$, it follows that $F$ maps into an open subset of $\mathcal{X}(\mathbb{S},\mathbb{M})$ having the structure of a (possibly non-Hausdorff) complex manifold. The main result of~\cite{AllegrettiBridgeland1} was the following.

\begin{theorem}[\cite{AllegrettiBridgeland1}, Theorem~1.1]
\label{thm:AllegrettiBridgeland}
Let $(\mathbb{S},\mathbb{M})$ be a marked bordered surface, and if $\mathbb{S}$ has genus~0, assume that $|\mathbb{M}|\geq3$. Then the map $F$ is holomorphic.
\end{theorem}

It was conjectured in~\cite{AllegrettiBridgeland1} that $F$ is in fact a local biholomorphism. We refer the reader to~\cite{AllegrettiBridgeland1} for more discussion of this point.

\section{Voros symbols as cluster coordinates}
\label{sec:VorosSymbolsAsClusterCoordinates}

\subsection{Connection formulas}

The asymptotic behavior of the WKB solutions along a generic trajectory depends on the direction in which the real part of $w(z)$ is increasing. In the following, we will write $\Psi_{\pm}^{(p,a)}=\mathcal{S}[\psi_{\pm}^{(p,a)}]$ for the Borel sum of the WKB solution normalized at a pole~$p$.

\begin{lemma}
\label{lem:canonicalcoordinatesubdominant}
Let $\beta$ be a generic trajectory connecting distinct poles $p_1$ and~$p_2$, and let $a$ be a zero of~$\phi$. Suppose that $\Re\left(\frac{1}{\hbar}\int_a^z\sqrt{Q_0(\zeta)}d\zeta\right)$ is increasing from~$p_1$ to~$p_2$ along~$\beta$. Then $\Psi_+^{(p_1,a)}$ is decreasing as $z\rightarrow p_1$ along~$\beta$, and $\Psi_-^{(p_2,a)}$ is decreasing as $z\rightarrow p_2$ along~$\beta$.
\end{lemma}

\begin{proof}
By definition of the Borel resummation, we have 
\[
\Psi_+^{(p_1,a)}=\exp\left(\frac{1}{\hbar}\int_a^z\sqrt{Q_0(\zeta)}d\zeta\right)\mathcal{S}[g^{(p_1)}]
\]
where 
\[
g^{(p_1)}(z,\hbar)=\frac{1}{\sqrt{S_{\text{odd}}(z,\hbar)}}\exp\left(\int_{p_1}^zS_{\text{odd}}^{\text{reg}}(\zeta,\hbar)d\zeta\right).
\]
The function $\mathcal{S}[g^{(p_1)}]$ is bounded as $z\rightarrow p_1$ along~$\beta$. Under our assumptions, the exponential factor in the expression for $\Psi_+^{(p_1,a)}$ is decreasing as $z\rightarrow p_1$, so the function $\Psi_+^{(p_1,a)}$ is decreasing. The second statement is proved by a similar argument.
\end{proof}

For any pole $p$, the functions $\Psi_{\pm}^{(p,a)}$ form a basis for the space of solutions of~\eqref{eqn:schrodinger}. We can now state connection formulas for transforming between the bases associated to different poles~$p$.

\begin{proposition}
\label{prop:horizontalstrip}
Let $D$ be a horizontal strip or half plane, and let $\beta$ be a generic trajectory in~$D$ connecting distinct poles $p_1$ and $p_2$ on the boundary of~$D$. Then we have 
\[
\Psi_+^{(p_1,a)}=x_{12}\Psi_+^{(p_2,a)} \quad \text{and} \quad \Psi_-^{(p_1,a)}=x_{12}^{-1}\Psi_-^{(p_2,a)}
\]
where 
\[
x_{12}=\mathcal{S}\left[\exp\left(\int_{p_1}^{p_2}S_{\text{odd}}^{\text{reg}}(z,\hbar)dz\right)\right].
\]
\end{proposition}

\begin{proof}
By the properties of Borel resummation, we have 
\begin{align*}
x_{12}\Psi_+^{(p_2,a)} &= \exp\left(\frac{1}{\hbar}\int_a^z\sqrt{Q_0(\zeta)}d\zeta\right)\mathcal{S}\biggl[\frac{1}{\sqrt{S_{\text{odd}}(z,\hbar)}}\exp\left(\int_{p_2}^zS_{\text{odd}}^{\text{reg}}(\zeta,\hbar)d\zeta\right) \\
& \qquad \cdot\exp\left(\int_{p_1}^{p_2}S_{\text{odd}}^{\text{reg}}(\zeta,\hbar)d\zeta\right)\biggr] \\
&= \exp\left(\frac{1}{\hbar}\int_a^z\sqrt{Q_0(\zeta)}d\zeta\right)\mathcal{S}\left[\frac{1}{\sqrt{S_{\text{odd}}(z,\hbar)}}\exp\left(\int_{p_1}^zS_{\text{odd}}^{\text{reg}}(\zeta,\hbar)d\zeta\right)\right] \\
&= \Psi_+^{(p_1,a)}
\end{align*}
and similarly for the other identity.
\end{proof}

The next result describes how solutions of the Schr\"odinger equation transform when we analytically continue across a separating trajectory.

\begin{proposition}
\label{prop:connectionformulaseparating}
Suppose $\phi$ is a saddle-free GMN differential. Let $D_1$ and $D_2$ be adjacent horizontal strips or half planes as illustrated in either of the diagrams below, and let $C$ be the trajectory separating these regions. Let $a$ be a simple zero of~$\phi$ and $p$ a pole as indicated in the diagrams.
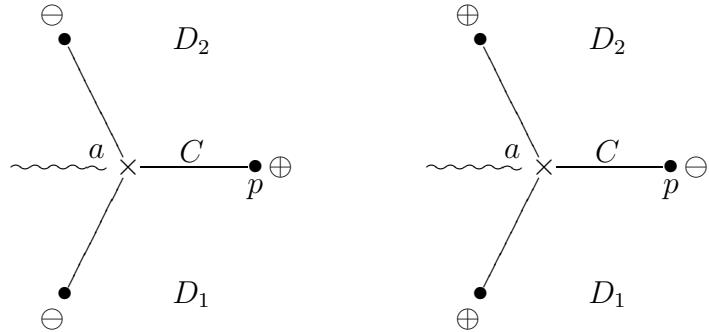
\[
\xy /l2pc/:
(0,-2)*{\bullet}="21";
(1,0)*{}="22";
(-1,0)*{\times}="32";
(-3,0)*{\bullet}="42";
(0,2)*{\bullet}="23";
{"22"\PATH~={**@{~}}'"32"},
{"32"\PATH~={**@{-}}'"42"},
{"21"\PATH~={**@{-}}'"32"},
{"32"\PATH~={**@{-}}'"23"},
(0.2,-2.4)*{\ominus};
(0.2,2.4)*{\ominus};
(-3.4,0)*{\oplus};
(-0.5,-0.25)*{a};
(-2,-0.25)*{C};
(-2,-2)*{D_2};
(-2,2)*{D_1};
(-3,0.35)*{p};
\endxy
\qquad
\qquad
\xy /l2pc/:
(0,-2)*{\bullet}="21";
(1,0)*{}="22";
(-1,0)*{\times}="32";
(-3,0)*{\bullet}="42";
(0,2)*{\bullet}="23";
{"22"\PATH~={**@{~}}'"32"},
{"32"\PATH~={**@{-}}'"42"},
{"21"\PATH~={**@{-}}'"32"},
{"32"\PATH~={**@{-}}'"23"},
(0.2,-2.4)*{\oplus};
(0.2,2.4)*{\oplus};
(-3.4,0)*{\ominus};
(-0.5,-0.25)*{a};
(-2,-0.25)*{C};
(-2,-2)*{D_2};
(-2,2)*{D_1};
(-3,0.35)*{p};
\endxy
\]
For $j=1,2$, write $\Psi_{D_j,\pm}^{(p,a)}(z,\hbar)$ for the Borel sum of $\psi_{\pm}^{(p,a)}(z,\hbar)$ in the region $D_j$. Then we have one of the following identities, depending on the signs.
\begin{enumerate}
\item If the signs are as in the diagram on the left, then the analytic continuation of $\Psi_{D_1,\pm}^{(p,a)}$ across the curve $C$ into the region $D_2$ satisfies 
\[
\Psi_{D_1,+}^{(p,a)}=\Psi_{D_2,+}^{(p,a)}+i\frac{e_-}{e_+}\Psi_{D_2,-}^{(p,a)}
\quad
\text{and}
\quad
\Psi_{D_1,-}^{(p,a)}=\Psi_{D_2,-}^{(p,a)}
\]
for some nonzero constants $e_{\pm}$.
\item If the signs are as in the diagram on the right, then the analytic continuation of $\Psi_{D_1,\pm}^{(p,a)}$ across the curve $C$ into the region $D_2$ satisfies 
\[
\Psi_{D_1,+}^{(p,a)}=\Psi_{D_2,+}^{(p,a)}
\quad
\text{and}
\quad
\Psi_{D_1,-}^{(p,a)}=\Psi_{D_2,-}^{(p,a)}+i\frac{e_+}{e_-}\Psi_{D_2,+}^{(p,a)}
\]
for some nonzero constants $e_{\pm}$.
\end{enumerate}
\end{proposition}

\begin{proof}
To prove this result, we employ a slightly different connection formula stated in Theorem~2.25 of~\cite{IwakiNakanishi}. This connection formula involves a different version of the WKB solutions. Namely, we define the WKB solution \emph{normalized at the zero} $a$ by the expression 
\[
\psi_{\pm}^{(a)}(z,\hbar)=\frac{1}{\sqrt{S_{\text{odd}}(z,\hbar)}}\exp\left(\pm\int_a^zS_{\text{odd}}(\zeta,\hbar)d\zeta\right).
\]
To define the integral appearing in this expression, we consider the contour $\gamma_z$ illustrated below.
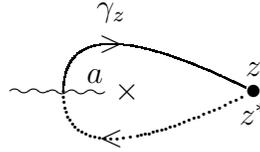
\[
\xy /l2pc/:
(1,0)*{}="22";
(-1,0)*{\times}="32";
(-3,0)*{\bullet}="42";
(0,0)*{}="m";
{"22"\PATH~={**@{~}}'"32"},
"m";"42" **\crv{(0,-1) & (-1,-1)}; 
"m";"42" **\crv{~*=<2pt>{.} (0,1) & (-1,1)}; 
(-0.5,-0.25)*{a};
(-3,-0.35)*{z};
(-3,0.35)*{z^*};
(-0.75,-0.75)*{>};
(-0.75,0.75)*{<};
(-0.75,-1.25)*{\gamma_z};
\endxy
\]
This contour begins at the point $z^*$ for some $z\in\widehat{S}$, proceeds along the second sheet until it reaches the branch cut, and then continues along the first sheet before terminating at~$z$. Its projection to~$S$ goes around~$a$ as illustrated. In terms of this contour, we can define the integral in the above expression as 
\[
\int_a^zS_{\text{odd}}(\zeta,\hbar)d\zeta=\frac{1}{2}\int_{\gamma_z}S_{\text{odd}}(\zeta,\hbar)d\zeta.
\]
Then 
\begin{align*}
\psi_{\pm}^{(a)} &(z,\hbar) \\
&= \frac{1}{\sqrt{S_{\text{odd}}(z,\hbar)}}\exp\left(\pm\frac{1}{2}\left(\int_{\gamma_z}S_{\text{odd}}^{\text{reg}}(\zeta,\hbar)d\zeta+\frac{1}{\hbar}\int_{\gamma_z}\sqrt{Q_0(\zeta)}d\zeta\right)\right) \\
&=\frac{1}{\sqrt{S_{\text{odd}}(z,\hbar)}}\exp\left(\pm\left(\int_p^z S_{\text{odd}}^{\text{reg}}(\zeta,\hbar)d\zeta + \frac{1}{\hbar}\int_a^z\sqrt{Q_0(\zeta)}d\zeta + \frac{1}{2}\int_{\gamma_p} S_{\text{odd}}^{\text{reg}}(\zeta,\hbar)d\zeta\right)\right) \\
&=\psi_{\pm}^{(p,a)}(z,\hbar)\exp\left(\pm\frac{1}{2}\int_{\gamma_p}S_{\text{odd}}^{\text{reg}}(\zeta,\hbar)d\zeta\right) = e_{\pm}\psi_{\pm}^{(p,a)}(z,\hbar)
\end{align*}
where we have written $e_{\pm}$ for the exponential factor in the last line. Thus the matrix that describes how the Borel sums of $e_{\pm}\psi_{\pm}^{(p,a)}(z,\hbar)$ transform under analytic continuation is the same as the one from Theorem~2.25 of~\cite{IwakiNakanishi}. This completes the proof.
\end{proof}

Finally, consider a horizontal strip $D$ with a zero $a$ and pole $p$ on its boundary. As before, we write $\Psi_{D,\pm}^{(p,a)}(z,\hbar)$ for the Borel sum of the WKB solution normalized at $p$ in the domain~$D$. The following result says how these solutions transform when we change the zero~$a$.

\begin{proposition}
\label{prop:changezero}
Let $b$ be a zero on the boundary of~$D$, and let $\alpha$ be a path from~$a$ to~$b$. Then we have 
\[
\Psi_{D,\pm}^{(p,a)}=\exp\left(\pm\frac{1}{\hbar}\int_\alpha\sqrt{Q_0(\zeta)}d\zeta\right)\Psi_{D,\pm}^{(p,b)}.
\]
\end{proposition}

\begin{proof}
By definition of the WKB solution normalized at a pole, we have 
\begin{align*}
\psi_\pm^{(p,a)}(z,\hbar) &= \frac{1}{\sqrt{S_{\text{odd}}(z,\hbar)}}\exp\left(\pm\left(\frac{1}{\hbar}\int_a^z\sqrt{Q_0(\zeta)}d\zeta + \int_p^zS_{\text{odd}}^{\text{reg}}(\zeta,\hbar)d\zeta\right)\right) \\
&= \frac{1}{\sqrt{S_{\text{odd}}(z,\hbar)}}\exp\left(\pm\left(\frac{1}{\hbar}\int_\alpha\sqrt{Q_0(\zeta)}d\zeta + \frac{1}{\hbar}\int_b^z\sqrt{Q_0(\zeta)}d\zeta + \int_p^zS_{\text{odd}}^{\text{reg}}(\zeta,\hbar)d\zeta\right)\right) \\
&= \exp\left(\pm\frac{1}{\hbar}\int_\alpha\sqrt{Q_0(\zeta)}d\zeta\right)\psi_\pm^{(p,b)}(z,\hbar).
\end{align*}
The statement follows by taking Borel sums.
\end{proof}

We can use Proposition~\ref{prop:changezero} to calculate the monodromy of a WKB solution around a pole of order two.

\begin{proposition}
\label{prop:eigenvectormonodromy}
Let $p$ be a pole of $\phi$ of order two and $a$ be a zero. Let $D$ a horizontal strip having $a$ and $p$ on its boundary. 
\begin{enumerate}
\item If the horizontal trajectories are oriented towards the preimage of $p$ in the first sheet of~$\Sigma_\phi$, then $\Psi_{D,-}^{(p,a)}$ is an eigenvector of the monodromy along a small loop $\beta_p$ encircling~$p$ with eigenvalue $\exp\left(-\frac{1}{\hbar}\int_{\beta_p}\sqrt{Q_0(\zeta)}d\zeta\right)$.
\item If the horizontal trajectories are oriented away from the preimage of $p$ in the first sheet of~$\Sigma_\phi$, then $\Psi_{D,+}^{(p,a)}$ is an eigenvector of the monodromy along a small loop $\beta_p$ encircling~$p$ with eigenvalue $\exp\left(\frac{1}{\hbar}\int_{\beta_p}\sqrt{Q_0(\zeta)}d\zeta\right)$.
\end{enumerate}
\end{proposition}

\begin{proof}
We prove only part~1 as the proof of part~2 is similar. If $D$ is a degenerate horizontal strip as in the diagram 
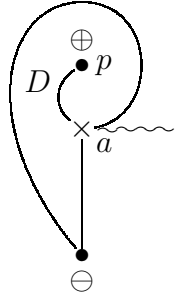
\[
\xy /l2pc/:
(0,0)*{\bullet}="3";
(0,1)*{\times}="4";
(0,3)*{\bullet}="5";
(0,-1)*{}="b";
(-1.5,1)*{}="c";
{"4"\PATH~={**@{-}}'"5"},
{"c"\PATH~={**@{~}}'"4"},
"3";"4" **\crv{(0.5,0.25) & (0.5,0.75)}; 
"b";"5" **\crv{(1.25,-1) & (1.75,1)}; 
"b";"4" **\crv{(-1.25,-1) & (-1.25,1)}; 
(0,-0.4)*{\oplus};
(0,3.4)*{\ominus};
(0.7,0.25)*{D};
(-0.35,0)*{p};
(-0.35,1.25)*{a};
\endxy
\]
then the claim follows immediately from Propositions~\ref{prop:connectionformulaseparating} and~\ref{prop:changezero}. Otherwise, let $D_0,\dots,D_n$ be the horizontal strips adjacent to $p$, ordered cyclically as in the diagram below with $D=D_0$ and $a=a_0$.
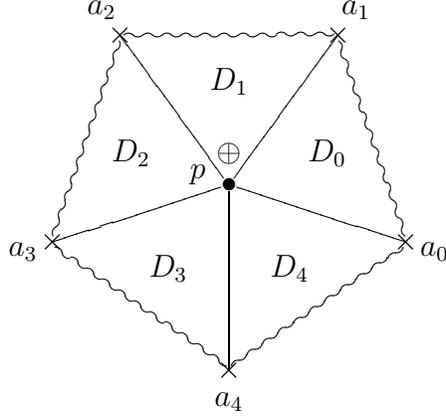
\[
\xy /l3.25pc/:
{\xypolygon5"B"{~:{(1.8,0):}~>{}}},
{\xypolygon5"E"{~:{(-1,0):}~>{}}},
{\xypolygon5"F"{~:{(2.1,0):}~>{}}},
(1,0)*{\bullet}="0";
{"0"\PATH~={**@{-}}"B1"},
{"0"\PATH~={**@{-}}"B2"},
{"0"\PATH~={**@{-}}"B3"},
{"0"\PATH~={**@{-}}"B4"},
{"0"\PATH~={**@{-}}"B5"},
{"B1"\PATH~={**@{~}}'"B2"},
{"B2"\PATH~={**@{~}}'"B3"},
{"B3"\PATH~={**@{~}}'"B4"},
{"B4"\PATH~={**@{~}}'"B5"},
{"B5"\PATH~={**@{~}}'"B1"},
"B1"*{\times},
"B2"*{\times},
"B3"*{\times},
"B4"*{\times},
"B5"*{\times},
"E1"*{D_0},
"E2"*{D_1},
"E3"*{D_2},
"E4"*{D_3},
"E5"*{D_4},
"F1"*{a_3},
"F2"*{a_4},
"F3"*{a_0},
"F4"*{a_1},
"F5"*{a_2},
(1.3,-0.1)*{p};
(1,-0.3)*{\oplus};
\endxy
\]
Let $\beta_p$ be the loop that starts at a point in~$D_0$ and goes around~$p$, passing through each of the regions $D_0,\dots,D_n$ in order before returning to the starting point. Let $a_0,\dots,a_n$ be zeros, labeled as in the diagram, and for each $i=0,\dots,n$, let $\alpha_i$ be a path in $D_i$ going from~$a_i$ to~$a_{i+1}$ where we number the indices modulo~$n+1$. By repeatedly applying Propositions~\ref{prop:connectionformulaseparating} and~\ref{prop:changezero}, we see that the analytic continuation of $\Psi_{D_0,-}^{(p,a_0)}$ along $\beta_p$ is given by 
\[
\exp\left(-\frac{1}{\hbar}\int_{\alpha_0}\sqrt{Q_0(\zeta)}d\zeta\right)\dots\exp\left(-\frac{1}{\hbar}\int_{\alpha_n}\sqrt{Q_0(\zeta)}d\zeta\right)\Psi_{D_0,-}^{(p,a_0)},
\]
which equals 
\[
\exp\left(-\frac{1}{\hbar}\int_{\beta_p}\sqrt{Q_0(\zeta)}d\zeta\right)\Psi_{D_0,-}^{(p,a_0)}
\]
as desired.
\end{proof}

\subsection{Calculation of cross ratios}
\label{sec:CalculationOfCrossRatios}

We will now use the connection formulas proved above to explain the relationship between Voros symbols and cross ratios. In what follows, we assume $\phi$ is a complete saddle-free differential associated with a differential equation of the form~\eqref{eqn:schrodinger}.

Consider any generic trajectory $j$ corresponding to an arc of the WKB triangulation. If $\mathcal{U}$ is the universal cover of $S\setminus\Pol(\phi)$, then we can lift $j$ to an arc $\tilde{j}$ in~$\mathcal{U}$. Let $\tilde{\phi}$ be the quadratic differential on~$\mathcal{U}$ obtained by pulling back the quadratic differential $\phi$ along the covering map. This quadratic differential $\tilde{\phi}$ induces a foliation of $\mathcal{U}$, and $\tilde{j}$ is contained in a regular horizontal strip $D_0$ of this foliation. We can label the nearby regions and critical points of $\tilde{\phi}$ as in the diagram on the left below.
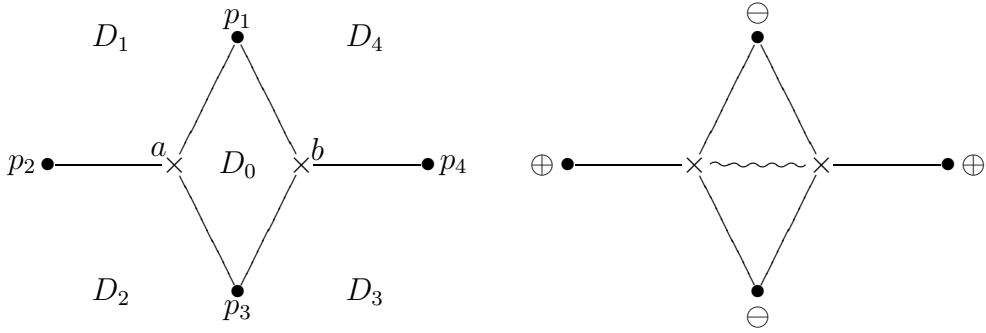
\[
\xy /l2pc/:
(0,-2)*{\bullet}="21";
(3,0)*{\bullet}="12";
(1,0)*{\times}="22";
(-1,0)*{\times}="32";
(-3,0)*{\bullet}="42";
(0,2)*{\bullet}="23";
{"12"\PATH~={**@{-}}'"22"},
{"32"\PATH~={**@{-}}'"42"},
{"21"\PATH~={**@{-}}'"22"},
{"21"\PATH~={**@{-}}'"32"},
{"22"\PATH~={**@{-}}'"23"},
{"32"\PATH~={**@{-}}'"23"},
(0,0)*{D_0};
(2,-2)*{D_1};
(2,2)*{D_2};
(-2,2)*{D_3};
(-2,-2)*{D_4};
(0,-2.3)*{p_1};
(3.4,0)*{p_2};
(0,2.3)*{p_3};
(-3.4,0)*{p_4};
(1.25,-0.25)*{a};
(-1.25,-0.25)*{b};
\endxy
\qquad
\xy /l2pc/:
(0,-2)*{\bullet}="21";
(3,0)*{\bullet}="12";
(1,0)*{\times}="22";
(-1,0)*{\times}="32";
(-3,0)*{\bullet}="42";
(0,2)*{\bullet}="23";
{"22"\PATH~={**@{~}}'"32"},
{"12"\PATH~={**@{-}}'"22"},
{"32"\PATH~={**@{-}}'"42"},
{"21"\PATH~={**@{-}}'"22"},
{"21"\PATH~={**@{-}}'"32"},
{"22"\PATH~={**@{-}}'"23"},
{"32"\PATH~={**@{-}}'"23"},
(0,-2.4)*{\ominus};
(0,2.4)*{\ominus};
(-3.4,0)*{\oplus};
(3.4,0)*{\oplus};
\endxy
\]
By abuse of notation, the same labels will also be used to denote the images of these critical points and regions in~$S$. The function $Q_0(z)$ on~$S$ can be pulled back to a function on $\mathcal{U}$ which we also denote $Q_0(z)$. We also consider the spectral cover $\Sigma_{\tilde{\phi}}$ of $\mathcal{U}$. Assume that the foliation of the first sheet of $\Sigma_{\tilde{\phi}}$ is oriented as in the diagram above on the right. Then we have the following statement.

\Needspace*{2\baselineskip}
\begin{lemma}\mbox{}
\label{lem:foursubdominants}
\begin{enumerate}
\item The analytic continuation of $\Psi_{D_1,+}^{(p_1,a)}$ across the horizontal strip $D_0$ into $D_4$ satisfies 
\[
\Psi_{D_1,+}^{(p_1,a)}=\exp\left(\frac{1}{\hbar}\int_{\gamma^+}\sqrt{Q_0(\zeta)}d\zeta\right)\Psi_{D_4,+}^{(p_1,b)}
\]
where $\gamma^+$ is a contour connecting $a$ to $b$ above the branch cut in the diagram.
\item The analytic continuation of $\Psi_{D_1,-}^{(p_2,a)}$ across the separating trajectory into $D_2$ satisfies 
\[
\Psi_{D_1,-}^{(p_2,a)}=\Psi_{D_2,-}^{(p_2,a)}.
\]
\item The analytic continuation of $\Psi_{D_3,+}^{(p_3,b)}$ across the horizontal strip $D_0$ into $D_2$ satisfies 
\[
\Psi_{D_3,+}^{(p_3,b)}=\exp\left(\frac{1}{\hbar}\int_{\gamma^-}\sqrt{Q_0(\zeta)}d\zeta\right)\Psi_{D_2,+}^{(p_3,a)}
\]
where $\gamma^-$ is a contour connecting $b$ to $a$ below the branch cut in the diagram.
\item The analytic continuation of $\Psi_{D_3,-}^{(p_4,b)}$ across the separating trajectory into $D_4$ satisfies 
\[
\Psi_{D_3,-}^{(p_4,b)}=\Psi_{D_4,-}^{(p_4,b)}.
\]
\end{enumerate}
\end{lemma}

\begin{proof}
This follows immediately from Propositions~\ref{prop:connectionformulaseparating} and~\ref{prop:changezero}.
\end{proof}

Let us denote the functions appearing in parts~1 through~4 of Lemma~\ref{lem:foursubdominants} by $\psi_1,\dots,\psi_4$, respectively. They are solutions of~\eqref{eqn:schrodinger}. If the pole $p_i$ has order at least three, then $\psi_i$ is subdominant near $p_i$ by Lemma~\ref{lem:canonicalcoordinatesubdominant}. On the other hand, if $p_i$ has order two, then $\psi_i$ is an eigenvector of the monodromy around $p_i$ by Proposition~\ref{prop:eigenvectormonodromy}. We will see later that these solutions in fact define the framing for a $PGL_2(\mathbb{C})$-local system naturally associated to~$\phi$. If we analytically continue the solutions $\psi_1,\dots,\psi_4$ to a common point in~$D_0$, we get corresponding vectors $v_1,\dots,v_4$ in the two-dimensional space of solutions of~\eqref{eqn:schrodinger}. Their cross ratio is given by the expression 
\[
Y_j=\frac{\det(v_1v_2)\det(v_3v_4)}{\det(v_2v_3)\det(v_1v_4)}
\]
where $\det(v_sv_t)$ denotes the determinant of the $2\times2$-matrix having columns $v_s$ and~$v_t$.

\begin{proposition}
\label{prop:regularcrossratio}
There is an identity 
\[
Y_j=\mathcal{S}[e^{V_{\gamma_{\tilde{j}}}}]
\]
where $\gamma_{\tilde{j}}$ is the cycle in $\Sigma_{\tilde{\phi}}$ associated to the arc $\tilde{j}$ as illustrated in the first diagram in Section~\ref{sec:FoliationOfTheSpectralCover}.
\end{proposition}

\begin{proof}
Consider any of the regions $D_i$ for $i=1,\dots,4$. The poles $p_i$ and $p_{i+1}$ lie on the boundary of this region if we number the poles modulo~4. Associated to these poles are the solutions $\psi_i$ and $\psi_{i+1}$ of~\eqref{eqn:schrodinger}, and if we analytically continue these to a common point in~$D_i$, we get a pair of vectors~$u_i$ and~$u_{i+1}$ in the two-dimensional space of solutions. By Propositions~\ref{prop:horizontalstrip}, \ref{prop:connectionformulaseparating}, and~\ref{prop:changezero}, we know that the matrices that describe how solutions transform under analytic continuation are elements of $SL_2(\mathbb{C})$. It follows that 
\[
\det(v_iv_{i+1})=\det(u_iu_{i+1}).
\]
Now, by Proposition~\ref{prop:horizontalstrip}, 
\[
\det(u_1u_2)=\det\left(\begin{array}{cc} 1 & 0 \\ 0 & x_{12} \end{array}\right)=x_{12}, \quad \det(u_3u_4)=\det\left(\begin{array}{cc} 1 & 0 \\ 0 & x_{34} \end{array}\right)=x_{34}.
\]
On the other hand, by Proposition~\ref{prop:horizontalstrip} and Lemma~\ref{lem:foursubdominants}, we have 
\[
\det(u_2u_3)=\det\left(\begin{array}{cc} 1 & 0 \\ 0 & x_{23}\exp\left(\frac{1}{\hbar}\int_{\gamma^-}\sqrt{Q_0(\zeta)}d\zeta\right) \end{array}\right) = x_{23}\exp\left(\frac{1}{\hbar}\int_{\gamma^-}\sqrt{Q_0(\zeta)}d\zeta\right)
\]
and 
\[
\det(u_1u_4)=\det\left(\begin{array}{cc} 1 & 0 \\ 0 & x_{14}\exp\left(\frac{1}{\hbar}\int_{\gamma^+}\sqrt{Q_0(\zeta)}d\zeta\right) \end{array}\right) = x_{14}\exp\left(\frac{1}{\hbar}\int_{\gamma^+}\sqrt{Q_0(\zeta)}d\zeta\right).
\]
Hence, by Lemma~7.1 of~\cite{IwakiNakanishi}, we have 
\[
Y_j=\exp\left(\frac{1}{\hbar}\oint_{\gamma_{\tilde{j}}}\sqrt{Q_0(\zeta)}d\zeta\right)\frac{x_{12}x_{34}}{x_{23}x_{14}} =\mathcal{S}[e^{V_{\gamma_{\tilde{j}}}}]
\]
as desired.
\end{proof}

If $j$ corresponds to an arc of the WKB triangulation which is not the internal edge of a self-folded triangle, we set $X_j=Y_j$. On the other hand, if $j$ is the internal edge of a self-folded triangle and $k$ is the edge that surrounds $j$, then we set $X_j=Y_jY_k$.

\begin{theorem}
For every arc $j$ of the WKB triangulation, there is an identity 
\[
X_j=\mathcal{S}[e^{V_{\gamma_j}}].
\]
\end{theorem}

\begin{proof}
If $j$ is not the internal edge of a self-folded triangle, then this follows immediately from Proposition~\ref{prop:regularcrossratio} and the fact that $V_{\gamma_{\tilde{j}}}=V_{\gamma_j}$. We therefore assume $j$ is the internal edge of a self-folded triangle. Suppose the cycle $\gamma_j$ is given by the picture 
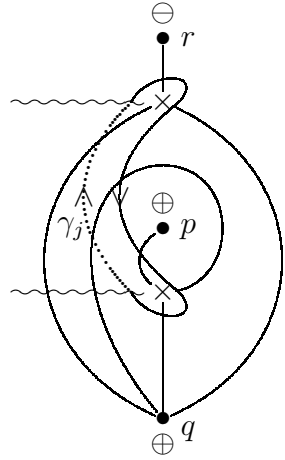
\[
\xy /l2pc/:
(0,-3)*{\bullet}="1";
(0,-2)*{\times}="2";
(0,0)*{\bullet}="3";
(0,1)*{\times}="4";
(0,3)*{\bullet}="5";
(2.5,-2)*{}="a";
(0,-1)*{}="b";
(2.5,1)*{}="c";
{"1"\PATH~={**@{-}}'"2"},
{"4"\PATH~={**@{-}}'"5"},
{"a"\PATH~={**@{~}}'"2"},
{"c"\PATH~={**@{~}}'"4"},
"2";"5" **\crv{(2.5,-1) & (2.5,2)}; 
"2";"5" **\crv{(-2.5,-1) & (-2.5,2)}; 
"3";"4" **\crv{(0.5,0.25) & (0.5,0.75)}; 
"b";"5" **\crv{(1.25,-1) & (1.75,1)}; 
"b";"4" **\crv{(-1.25,-1) & (-1.25,1)}; 
(0,-0.4)*{\oplus};
(0,-3.4)*{\ominus};
(0,3.4)*{\oplus};
(0.5,-2)*{}="x";
(-0.25,-2)*{}="y";
(0.5,1)*{}="z";
(-0.25,1)*{}="w";
"x";"y" **\crv{(0.25,-2.5) & (-0.7,-2.5)}; 
"y";"w" **\crv{(1,-1) & (1,0)}; 
"z";"w" **\crv{(0.25,1.5) & (-0.7,1.5)}; 
"x";"z" **\crv{~*=<2pt>{.} (1.5,-1) & (1.5,0)}; 
(1.25,-0.5)*{\wedge};
(0.68,-0.5)*{\vee};
(1.45,0)*{\gamma_j};
(-0.4,-3)*{r};
(-0.4,0)*{p};
(-0.4,3.2)*{q};
\endxy
\]
(The proof in the other case is similar.) Then the corresponding picture in the universal cover is the following:
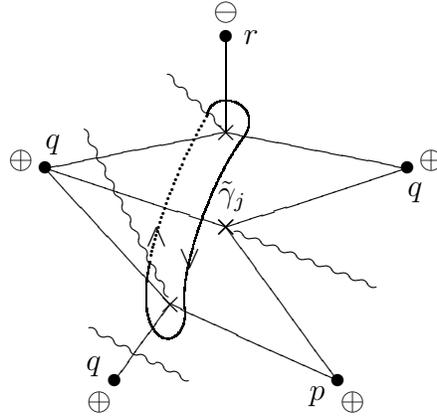
\[
\xy /l2pc/:
{\xypolygon5"A"{~:{(-3,0):}~>{}}},
{\xypolygon5"B"{~:{(-1.5,0):}~>{}}},
{\xypolygon5"C"{~:{(-3.4,0):}~>{}}},
{\xypolygon5"D"{~:{(-3,0.4):}~>{}}},
{"A1"\PATH~={**@{-}}"B2"},
{"A2"\PATH~={**@{-}}"B2"},
{"A3"\PATH~={**@{-}}"B2"},
{"A3"\PATH~={**@{-}}(1,0)},
{"A1"\PATH~={**@{-}}(1,0)},
{"A3"\PATH~={**@{-}}"B4"},
{"A5"\PATH~={**@{-}}"B4"},
{"A5"\PATH~={**@{-}}(1,0)},
{"A4"\PATH~={**@{-}}"B4"},
"A1"*{\bullet},
"A2"*{\bullet},
"A3"*{\bullet},
"A4"*{\bullet},
"A5"*{\bullet},
"B2"*{\times},
"B4"*{\times},
"C1"*{\oplus},
"C2"*{\ominus},
"C3"*{\oplus},
"C4"*{\oplus},
"C5"*{\oplus},
(1,0)*{\times},
"D1"*{q},
"D2"*{r},
"D3"*{q},
"D4"*{q},
"D5"*{p},
(1,0)*{\times},
(2,-2.5)*{}="b1";
(3.25,-1.65)*{}="b2";
(3.25,1.5)*{}="b3";
(1.5,2.5)*{}="b4";
(-1.5,1)*{}="b5";
{"b1"\PATH~={**@{~}}'"B2"},
{"b2"\PATH~={**@{~}}'"B4"},
{"b3"\PATH~={**@{~}}'"b4"},
{(1,0)\PATH~={**@{~}}'"b5"},
(1.3,-1.75)*{}="c1";
(0.7,-1.45)*{}="c2";
(2.2,0.55)*{}="c3";
(1.65,1.3)*{}="c4";
"c1";"c3" **\crv{~*=<2pt>{.} (1.75,-1) & (2.1,0)}; 
"c1";"c2" **\crv{(1.25,-2.2) & (0.4,-2)}; 
"c2";"c4" **\crv{(1.25,-0.75) & (1.75,0.75)}; 
"c3";"c4" **\crv{(2.5,2) & (1.6,2)}; 
(1.57,0.5)*{\vee};
(2.1,0.15)*{\wedge};
(0.9,-0.5)*{\tilde{\gamma}_j};
\endxy
\]
Here $\tilde{\gamma}_j$ denotes a lift of the cycle $\gamma_j$ to $\Sigma_{\tilde{\phi}}$, and we have labeled the poles by their images in~$S$. It is easy to see that, up to cycles invariant under the covering involution, $\tilde{\gamma}_j$ is the sum of two cycles, $\tilde{\gamma}_j^{(1)}$ and $\tilde{\gamma}_j^{(2)}$, illustrated in the diagram below.
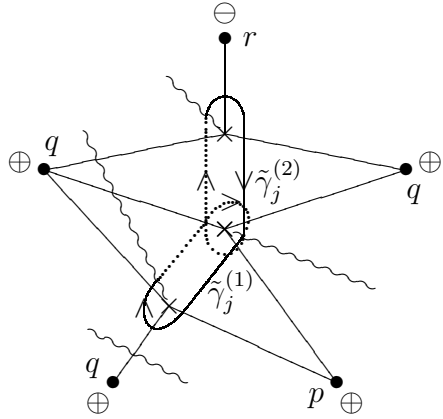
\[
\xy /l2pc/:
{\xypolygon5"A"{~:{(-3,0):}~>{}}},
{\xypolygon5"B"{~:{(-1.5,0):}~>{}}},
{\xypolygon5"C"{~:{(-3.4,0):}~>{}}},
{\xypolygon5"D"{~:{(-3,0.4):}~>{}}},
{"A1"\PATH~={**@{-}}"B2"},
{"A2"\PATH~={**@{-}}"B2"},
{"A3"\PATH~={**@{-}}"B2"},
{"A3"\PATH~={**@{-}}(1,0)},
{"A1"\PATH~={**@{-}}(1,0)},
{"A3"\PATH~={**@{-}}"B4"},
{"A5"\PATH~={**@{-}}"B4"},
{"A5"\PATH~={**@{-}}(1,0)},
{"A4"\PATH~={**@{-}}"B4"},
"A1"*{\bullet},
"A2"*{\bullet},
"A3"*{\bullet},
"A4"*{\bullet},
"A5"*{\bullet},
"B2"*{\times},
"B4"*{\times},
"C1"*{\oplus},
"C2"*{\ominus},
"C3"*{\oplus},
"C4"*{\oplus},
"C5"*{\oplus},
(1,0)*{\times},
"D1"*{q},
"D2"*{r},
"D3"*{q},
"D4"*{q},
"D5"*{p},
(1,0)*{\times},
(2,-2.5)*{}="b1";
(3.25,-1.65)*{}="b2";
(3.25,1.5)*{}="b3";
(1.5,2.5)*{}="b4";
(-1.5,1)*{}="b5";
{"b1"\PATH~={**@{~}}'"B2"},
{"b2"\PATH~={**@{~}}'"B4"},
{"b3"\PATH~={**@{~}}'"b4"},
{(1,0)\PATH~={**@{~}}'"b5"},
(1.3,-1.75)*{}="c1";
(0.7,-1.75)*{}="c2";
(0.7,0.1)*{}="c3";
(1.3,0.1)*{}="c4";
(2.05,0.8)*{}="c5";
(1.65,1.3)*{}="c6";
(1.3,-0.1)*{}="c7";
"c1";"c2" **\crv{(1.25,-2.2) & (0.75,-2.2)}; 
"c2";"c3" **\crv{"c2" & "c3"}; 
"c4";"c3" **\crv{~*=<2pt>{.} (1.25,0.5) & (0.75,0.5)}; 
"c1";"c4" **\crv{~*=<2pt>{.} "c1" & "c4"}; 
"c5";"c6" **\crv{(2.5,1.2) & (2.25,2)}; 
"c6";"c3" **\crv{"c6" & "c3"}; 
"c5";"c7" **\crv{~*=<2pt>{.} "c5" & "c7"}; 
"c7";"c3" **\crv{~*=<2pt>{.} (1.1,-0.6) & (0.4,-0.5)}; 
(0.7,-0.75)*{\vee};
(1.3,-0.75)*{\wedge};
(2.27,1.25)*{\wedge};
(0.9,-0.42)*{>};
(0.9,1)*{\tilde{\gamma}_j^{(1)}};
(0.15,-0.75)*{\tilde{\gamma}_j^{(2)}};
\endxy
\]
Let $k$ be the edge of the WKB triangulation that surrounds $j$. By changing the branch, one sees that $\tilde{\gamma}_j^{(1)}$ and $\tilde{\gamma}_j^{(2)}$ are equal, up to cycles invariant under the covering involution, to $\gamma_{\tilde{j}}$ and $\gamma_{\tilde{k}}$, respectively. Thus we have 
\begin{align*}
\mathcal{S}[e^{V_{\gamma_j}}] &= \mathcal{S}\left[\exp\left(\int_{\gamma_j} S_{\text{odd}}(z,\hbar)dz\right)\right] \\
&= \mathcal{S}[e^{V_{\gamma_{\tilde{j}}}}] \mathcal{S}[e^{V_{\gamma_{\tilde{k}}}}] \\
&= Y_j Y_k
\end{align*}
by Proposition~\ref{prop:regularcrossratio}. This completes the proof.
\end{proof}

\subsection{The main theorem}

Let $S$ be a closed Riemann surface. For any collection of distinct points $p_1,\dots,p_k$ on~$S$, we define $\mathcal{P}$ to be the projective structure on the punctured surface 
\[
S^*=S\setminus\{p_1,\dots,p_k\}
\]
obtained by uniformization.

\begin{lemma}
\label{lem:uniformization}
Let $\mathcal{P}_0$ be the projective structure obtained by uniformization of the closed surface~$S$. If $\chi(S^*)\leq0$, then for any of the points $p_i$, there is a chart $z$ in~$\mathcal{P}_0$ with $z(p_i)=0$ such that the quadratic differential $\mathcal{P}-\mathcal{P}_0|_{S^*}$ is given in this local coordinate by the expression 
\[
-\frac{1}{4z^2}dz^{\otimes2}.
\]
In particular, $\mathcal{P}$ is a meromorphic projective structure on~$S$.
\end{lemma}

\begin{proof}
To start, let us assume that $\chi(S^*)<0$. Then the uniformization theorem implies that the universal cover of $S^*$ is identified with the upper half plane. Under this identification, the deck transformation corresponding to a small loop surrounding $p_i$ is parabolic, and we may assume it is given by $z\mapsto z+1$. Thus there is a neighborhood $U$ of~$p_i$ in~$S$ such that the punctured neighborhood $U\setminus\{p_i\}$ is isomorphic as a complex manifold to the quotient 
\[
\{z:\Im(z)>h\}/\sim
\]
for some $h>0$ where $\sim$ is the relation defined by $z\sim z+1$. By applying the map $z\mapsto\exp(2\pi iz)$, we see that 
\[
U\setminus\{p_i\}\cong\{z:0<|z|<e^{-2\pi h}\}.
\]
A chart of $\mathcal{P}$ is defined locally on this set by $f(z)=\frac{1}{2\pi i}\log(z)$ for some branch of the logarithm. We want to compare this with the standard holomorphic projective structure on 
\[
U\cong\{z:0\leq|z|<e^{-2\pi h}\}.
\]
To do this we calculate the Schwarzian derivative $\{f(z),z\}$. It is easy to check that we have $-\{f(z),z\}dz^{\otimes2}/2=-dz^{\otimes2}/4z^2$.

Next, suppose that $\chi(S^*)=0$. If the set $\{p_1,\dots,p_k\}$ is nonempty, then the surface $S^*$ is necessarily a twice-punctured sphere. We can think of this Riemann surface as the complex plane punctured at $p_i=0$. Then the universal cover of $S^*$ is the Riemann surface on which the logarithm function is single-valued, and the charts of $\mathcal{P}$ are given by different branches of $f(z)=\log(z)$. Computing the Schwarzian derivative $\{f(z),z\}$ as before, we see that the difference $\mathcal{P}-\mathcal{P}_0|_{S^*}$ is given near any puncture by the expression $-dz^{\otimes2}/4z^2$.
\end{proof}

We will be interested in points of $\Proj^*(\mathbb{S},\mathbb{M})$ obtained by adding a meromorphic quadratic differential to a base projective structure of the type described in Lemma~\ref{lem:uniformization}.

\begin{definition}
Let $\phi$ be a meromorphic quadratic differential on a closed Riemann surface~$S$.
\begin{enumerate}
\item A \emph{marking} of $\phi$ by a marked bordered surface $(\mathbb{S},\mathbb{M})$ is an isotopy class of isomorphisms between $(\mathbb{S},\mathbb{M})$ and the marked bordered surface defined by~$\phi$.
\item A \emph{signing} for $\phi$ is defined to be a choice of sign for the residue at each pole of order two.
\end{enumerate}
\end{definition}

Let $\phi$ be a signed GMN differential on~$S$ with a marking by $(\mathbb{S},\mathbb{M})$, and assume that $\phi$ is complete and saddle-free so that $\phi$ determines an ideal triangulation of~$(\mathbb{S},\mathbb{M})$. Let $p_1,\dots,p_k$ be the poles of $\phi$ in~$S$. If $S=\mathbb{CP}^1$ and $k=1$, let us take $\mathcal{P}$ to be the standard holomorphic projective structure on~$S$. Otherwise, let $\mathcal{P}$ be the meromorphic projective structure provided by Lemma~\ref{lem:uniformization}. Consider the meromorphic projective structure given by 
\[
\mathcal{P}(\hbar)=\mathcal{P}+\frac{1}{\hbar^2}\phi
\]
for $\hbar\in\mathbb{C}^*$. Since $\phi$ is equipped with a marking, there is a natural marking for $\mathcal{P}(1)$. There is a local system of sets over~$\mathbb{C}^*$ whose fiber over $\hbar\in\mathbb{C}^*$ parametrizes the set of all markings for~$\mathcal{P}(\hbar)$. Using the flat connection of this local system, we get a marking for any~$\mathcal{P}(\hbar)$. Thus we can view $\mathcal{P}(\hbar)$ as a point of the space $\Proj(\mathbb{S},\mathbb{M})$. Note that this point will have monodromy as $\hbar$ goes around the origin.

Let $\mathcal{P}_0$ be the uniformizing projective structure for the closed surface~$S$. If $S=\mathbb{CP}^1$ and $k=1$, then we have 
\[
\mathcal{P}(\hbar)-\mathcal{P}_0=\frac{1}{\hbar^2}\phi.
\]
If we write $\phi(z)=Q_0(z)dz^{\otimes2}$ for some meromorphic function~$Q_0(z)$, then the charts of $\mathcal{P}(\hbar)$ are given by ratios of solutions to the differential equation 
\[
\hbar^2y''(z)-Q_0(z)y(z)=0,
\]
and we define $Q(z,\hbar)\coloneqq Q_0(z)$. If we do not have $S=\mathbb{CP}^1$ and $k=1$, then it follows from Lemma~\ref{lem:uniformization} that we have 
\[
\mathcal{P}(\hbar)-\mathcal{P}_0=\frac{1}{\hbar^2}\phi+\tilde{\phi}
\]
where the quadratic differential $\tilde{\phi}$ can be written in the neighborhood of a pole $p_i$ of order two as $\tilde{\phi}(z)=Q_2(z)dz^{\otimes2}$ with 
\[
Q_2(z)=-\frac{1}{4z^2}.
\]
In this case, the charts of~$\mathcal{P}(\hbar)$ are given in a neighborhood of $p_i$ by ratios of solutions to the differential equation 
\[
\label{eqn:explicitpotential}
\hbar^2y''(z)-\left(Q_0(z)+\hbar^2Q_2(z)\right)y(z)=0,
\]
and we define $Q(z,\hbar)\coloneqq Q_0(z)+\hbar^2Q_2(z)$. In each case, we see that the potential function $Q(z,\hbar)$ satisfies Assumption~\ref{assumption2.5}.

\begin{lemma}
\label{lem:eigenvalues}
Let $p$ be a pole of $\phi$ of order two, and let $\beta_p$ be a small loop surrounding~$p$. If $z$ is a local coordinate with $z(p)=0$ as above, then the monodromy of the differential equation $\hbar^2y''(z)-Q(z,\hbar)y(z)=0$ around $\beta_p$ has eigenvalues 
\[
-\exp\left(\pm\frac{1}{\hbar}\int_{\beta_p}\sqrt{Q_0}\right).
\]
\end{lemma}

\begin{proof}
In~\cite{AllegrettiBridgeland1}, we showed that the eigenvalues of the monodromy are given by $-\exp(\pm r(p)/2)$ where 
\[
r(p)=\pm2\pi i\sqrt{1+4\lim_{z\rightarrow0}z^2\frac{Q(z,\hbar)}{\hbar^2}}.
\]
We can write 
\[
Q_0(z)=az^{-2}\left(1+O(z)\right) \quad \text{as $z\rightarrow0$},
\]
and then we have 
\[
\lim_{z\rightarrow0}z^2\frac{Q(z,\hbar)}{\hbar^2}=\frac{1}{\hbar^2}a-\frac{1}{4}
\]
so that $r(p)=\pm4\pi i\sqrt{a}/\hbar$. The lemma now follows from the residue theorem.
\end{proof}

Next, we consider the horizontal foliation defined by~$\phi$. Since $\phi$ has simple zeros, there are exactly three separating trajectories emanating from each zero. Let us choose a branch cut between each pair of consecutive separating trajectories so that these branch cuts connect the zeros on opposite sides of each horizontal strip. The diagram below shows an example of the resulting picture in a neighborhood of a pole $p$ of order two.
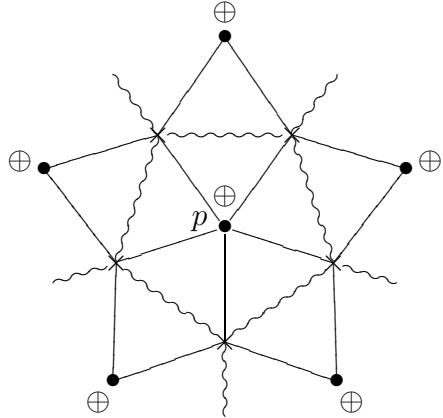
\[
\xy /l2pc/:
{\xypolygon5"A"{~:{(-3,0):}~>{}}},
{\xypolygon5"B"{~:{(1.8,0):}~>{}}},
{\xypolygon5"C"{~:{(-3.4,0):}~>{}}},
{\xypolygon5"D"{~:{(3,0):}~>{}}},
{\xypolygon5"E"{~:{(-1.5,0):}~>{}}},
{\xypolygon5"F"{~:{(2.2,0.3):}~>{}}},
{"A5"\PATH~={**@{-}}"B2"},
{"A5"\PATH~={**@{-}}"B3"},
{"A4"\PATH~={**@{-}}"B1"},
{"A4"\PATH~={**@{-}}"B2"},
{"A3"\PATH~={**@{-}}"B5"},
{"A3"\PATH~={**@{-}}"B1"},
{"A2"\PATH~={**@{-}}"B4"},
{"A2"\PATH~={**@{-}}"B5"},
{"A1"\PATH~={**@{-}}"B3"},
{"A1"\PATH~={**@{-}}"B4"},
(1,0)*{\bullet}="0";
{"0"\PATH~={**@{-}}"B1"},
{"0"\PATH~={**@{-}}"B2"},
{"0"\PATH~={**@{-}}"B3"},
{"0"\PATH~={**@{-}}"B4"},
{"0"\PATH~={**@{-}}"B5"},
{"B1"\PATH~={**@{~}}'"D1"},
{"B2"\PATH~={**@{~}}'"D2"},
{"B3"\PATH~={**@{~}}'"D3"},
{"B4"\PATH~={**@{~}}'"D4"},
{"B5"\PATH~={**@{~}}'"D5"},
{"B1"\PATH~={**@{~}}'"B2"},
{"B2"\PATH~={**@{~}}'"B3"},
{"B3"\PATH~={**@{~}}'"B4"},
{"B4"\PATH~={**@{~}}'"B5"},
{"B5"\PATH~={**@{~}}'"B1"},
"A1"*{\bullet},
"A2"*{\bullet},
"A3"*{\bullet},
"A4"*{\bullet},
"A5"*{\bullet},
"B1"*{\times},
"B2"*{\times},
"B3"*{\times},
"B4"*{\times},
"B5"*{\times},
"C1"*{\oplus},
"C2"*{\oplus},
"C3"*{\oplus},
"C4"*{\oplus},
"C5"*{\oplus},
(1.4,-0.1)*{p};
(1,-0.5)*{\oplus};
\endxy
\]
In such a neighborhood, let us take the first sheet of the spectral cover $\Sigma_\phi$ to be the sheet determined by the $\oplus$ sign, that is, the sheet where all trajectories are oriented into the pole. The local coordinate on this first sheet is a function which we will take as our $\sqrt{Q_0}$.

Since we assume that $\phi$ is equipped with a signing, we have a choice of sign for the residue at each pole of order two. Note that the residue at a pole $p$ is equal to $\pm2\int_{\beta_p}\sqrt{Q_0}$ where $\beta_p$ is a small counterclockwise oriented loop surrounding~$p$.

\begin{definition}
Let $\phi$ be a complete saddle-free signed GMN differential. Then the \emph{signed WKB triangulation} $(T,\epsilon)$ is defined as the WKB triangulation $T$ of~$\phi$ equipped with the signing $\epsilon:\mathbb{P}\rightarrow\{\pm1\}$ defined by the following rule: We set $\epsilon(p)=+1$ if the residue at $p$ is $-2\int_{\beta_p}\sqrt{Q_0}$ and set $\epsilon(p)=-1$ otherwise.
\end{definition}

Note that we could have defined the function $\sqrt{Q_0}$ by taking the $\ominus$ sign at each pole of order two in the above diagram. In this case, one should set $\epsilon(p)=+1$ if the residue at $p$ is $2\int_{\beta_p}\sqrt{Q_0}$ and set $\epsilon(p)=-1$ otherwise. It is easy to check that our definition of the signed WKB triangulation is equivalent to the definition given in Section~10 of~\cite{BridgelandSmith}.

We can now state and prove our main theorem.

\begin{theorem}
\label{thm:main}
There exists $\varepsilon>0$ such that 
\begin{enumerate}
\item For all points $\hbar\in\mathbb{H}(\varepsilon)$, the Fock-Goncharov coordinates of $F(\mathcal{P}(\hbar))$ with respect to the signed WKB triangulation of~$\phi$ are well defined.
\item Taking the Fock-Goncharov coordinate associated to an arc of the signed WKB triangulation gives a holomorphic map $\mathbb{H}(\varepsilon)\rightarrow\mathbb{C}^*$ which agrees with the Borel sum of the corresponding Voros symbol.
\end{enumerate}
\end{theorem}

\begin{proof}
The choice of signed differential determines an eigenvalue of the monodromy around a pole of order two by Lemma~\ref{lem:eigenvalues}. This defines the framing of $F(\mathcal{P}(\hbar))$ near each pole of order two. Near a pole of order $\geq3$, the framing is defined by subdominant solutions as we have explained. Consider the framed local system $\epsilon\cdot F(\mathcal{P}(\hbar))$ where we regard the signing~$\epsilon$ as an element of the group $(\mathbb{Z}/2\mathbb{Z})^{\mathbb{P}}$. By construction, the framing near any puncture is an eigenline of the monodromy around the puncture with eigenvalue $\exp\left(-\frac{1}{\hbar}\int_{\beta_p}\sqrt{Q_0}\right)$. By Proposition~\ref{prop:eigenvectormonodromy}, the Borel sum of the WKB solution $\psi_-$ is a nonzero vector in this eigenline. Similarly, by Lemma~\ref{lem:canonicalcoordinatesubdominant}, the Borel sum of $\psi_-$ is subdominant near a pole of order $\geq3$. Note that the framing line determined by $\psi_-$ is the same as the line determined by $\psi_+$ if we change the branch by replacing a $\oplus$ sign by $\ominus$. The theorem therefore follows from the computations of Subsection~\ref{sec:CalculationOfCrossRatios}.
\end{proof}

To prove Theorem~\ref{thm:meromorphic} from the introduction, we begin with some lemmas.

\begin{lemma}
\label{lem:discreteset}
Let $\mathcal{V}$ be the set of all $\hbar\in\mathbb{C}^*$ such that $\mathcal{P}(\hbar)$ has an apparent singularity at some pole of the differential~$\phi$. Then $\mathcal{V}$ has no accumulation point in~$\mathbb{C}^*$.
\end{lemma}

\begin{proof}
Let $p$ be a pole of the differential $\phi$. If $\beta_p$ is a small loop surrounding~$p$, then the monodromy around $\beta_p$, normalized to lie in $SL_2(\mathbb{C})$, has eigenvalues given by Lemma~\ref{lem:eigenvalues}. If $p$ is an apparent singularity of~$\mathcal{P}(\hbar)$ so that $\hbar\in\mathcal{V}$, then these eignvalues must equal $\pm1$. It is clear from the expression in Lemma~\ref{lem:eigenvalues} that if we perturb $\hbar$ by a small amount, then the eigenvalues will no longer have this property. The lemma follows.
\end{proof}

\begin{lemma}
\label{lem:extendfunction}
For each arc $j$ of the signed WKB triangulation of~$\phi$, the function 
\[
\mathcal{Y}_j:\hbar\mapsto X_j(F(\mathcal{P}(\hbar)))
\]
is a multivalued meromorphic continuation of $\mathcal{S}[e^{V_j}]$ from $\mathbb{H}(\varepsilon)$ to~$\mathbb{C}^*\setminus\mathcal{V}$ where $\mathcal{V}$ is the discrete set defined in Lemma~\ref{lem:discreteset}. It is branched only at the origin.
\end{lemma}

\begin{proof}
By Theorem~\ref{thm:main}, we know that these functions coincide on $\mathbb{H}(\varepsilon)$, so we only need to check that $\mathcal{Y}_j$ is meromorphic near any point of~$\mathbb{C}^*\setminus\mathcal{V}$. Consider any point $\hbar\in\mathbb{C}^*\setminus\mathcal{V}$ and choose a path $\gamma:[0,1]\rightarrow\mathbb{C}^*\setminus\mathcal{V}$ connecting $\mathbb{H}(\varepsilon)$ to this point. Any point in the image of~$F$ is a regular point of some coordinate chart, so for every point $\gamma(t)$ on the path, there is some tagged triangulation~$\tau_t$ such that $X_{\tau_t}$ is regular at $F(\mathcal{P}(\gamma(t)))$. In particular, the map $X_{\tau_t}\circ F\circ\mathcal{P}$ is holomorphic in a neighborhood $U_t$ of $\gamma(t)$ in~$\mathbb{C}^*\setminus\mathcal{V}$. These sets $U_t$ define an open cover of $\gamma([0,1])$, and since this image is compact, we can choose finitely many points 
\[
0=t_0<t_1<t_2<\dots<t_N=1
\]
so that $\{U_{t_i}\}$ is an open cover of $\gamma([0,1])$. We can assume that $U_{t_0}$ is the domain $\mathbb{H}(\varepsilon)$ and $U_{t_N}$ is a neighborhood of~$\hbar$. Since the transition maps for the Fock-Goncharov coordinates are birational, we can write $X_j$ as a rational function 
\[
X_j=\frac{p_i(X_1^{(i)},\dots,X_n^{(i)})}{q_i(X_1^{(i)},\dots,X_n^{(i)})}
\]
where $X_1^{(i)},\dots,X_n^{(i)}$ are the Fock-Goncharov coordinates associated to the tagged triangulation $\tau_{t_i}$. The domain $U_{t_1}$ intersects $U_{t_0}=\mathbb{H}(\varepsilon)$. Since $\mathcal{Y}_j$ is holomorphic on $\mathbb{H}(\varepsilon)$, we know that $X_j$ is regular on $F(\mathcal{P}(U_{t_0}\cap U_{t_1}))$. It follows that $q_1$ does not vanish identically on $F(\mathcal{P}(U_{t_1}))$. Therefore $X_j$ is regular on a nonempty open subset of $F(\mathcal{P}(U_{t_1}\cap U_{t_2}))$, and $q_2$ does not vanish identically on $F(\mathcal{P}(U_{t_2}))$. Continuing in this manner, we see that $q_N$ does not vanish identically on $F(\mathcal{P}(U_{t_N}))$. This means that $\mathcal{Y}_j$ is expressible as a ratio of holomorphic functions on~$U_{t_N}$ where the denominator has isolated zeros. Hence $\mathcal{Y}_j$ is meromorphic in a neighborhood of~$\hbar$. The multivaluedness of this function arises from the multivaluedness of $\mathcal{P}(\hbar)\in\Proj(\mathbb{S},\mathbb{M})$.
\end{proof}

The statement of Theorem~\ref{thm:meromorphic} is thus equivalent to the following.

\begin{theorem}
For each arc $j$ of the signed WKB triangulation of~$\phi$, the function $\mathcal{Y}_j$ extends to a multivalued meromorphic function on~$\mathbb{C}^*$, branched only at the origin.
\end{theorem}

\begin{proof}
Since $\mathcal{V}$ is discrete, we can choose a disk $D$ around any~$\hbar_0\in\mathcal{V}$ such that $D$ contains no other point of~$\mathcal{V}$. Consider a pole $p$ of $\phi$ of order two. The monodromy $M$ around~$p$ depends holomorphically on $\hbar\in D$. For $\hbar\in D\setminus\{\hbar_0\}$, let $\lambda(\hbar)$ denote the distinguished eigenvalue of the monodromy determined by the signing. Then by Lemma~\ref{lem:eigenvalues}, $\lambda$ extends to a holomorphic function on the domain~$D$. It follows that the entries of the matrix 
\[
M(\hbar)-\lambda(\hbar)=\left(\begin{array}{cc} a(\hbar) & b(\hbar) \\ c(\hbar) & d(\hbar)\end{array}\right)
\]
are holomorphic functions on~$D$. We claim that the diagonal entries of this matrix are not both identically zero. Indeed, for a general $\hbar\in D\setminus\{\hbar_0\}$ the remarks following Lemma~5.1 of~\cite{AllegrettiBridgeland1} imply that monodromy matrix $M(\hbar)$ is diagonal with distinct eigenvalues. It follows that after subtracting the scalar matrix $\lambda(\hbar)$ we get a matrix which again has distinct diagonal entries. This proves the claim. Now the $\lambda(\hbar)$-eigenspace of $M(\hbar)$ is 1-dimensional, and a vector $v=(v_1(\hbar),v_2(\hbar))^t$ in this eigenspace satisfies the conditions 
\[
a(\hbar)v_1(\hbar)+b(\hbar)v_2(\hbar)=0, \quad c(\hbar)v_1(\hbar)+d(\hbar)v_2(\hbar)=0.
\]
Let $\ell_p:D\setminus\{\hbar_0\}\rightarrow\mathbb{CP}^1$ be the map giving the distinguished eigenline at~$p$. Then there is an affine chart on~$\mathbb{CP}^1$ in which $\ell_p$ is given by one of the formulas $\ell(\hbar)=-b(\hbar)/a(\hbar)$ or $\ell(\hbar)=-c(\hbar)/d(\hbar)$, depending on which of the denominators is not identically zero. Thus $\ell_p$ extends to a holomorphic map $\ell_p:D\rightarrow\mathbb{CP}^1$ given by the same formulas. We can therefore associate a framed local system $\mathcal{F}(\hbar)$ to each $\hbar\in D$. This framed local system is defined as the monodromy local system of the projective structure $\mathcal{P}(\hbar)$ with a framing given by the lines $\ell_p(\hbar)$ near a pole $p$ of order two, and given by subdominant solutions near a higher order pole. In particular, we have $\mathcal{F}(\hbar)=F(\mathcal{P}(\hbar))$ for all $\hbar\in D\setminus\{\hbar_0\}$. Now the cluster coordinate $X_j(\mathcal{F}(\hbar))$ is given in terms of cross ratios of the framing lines which depend holomorphically on~$\hbar$ by the above argument. Therefore it is a ratio of holomorphic functions and has at worst a pole at~$\hbar_0$.
\end{proof}

Finally, we prove Theorem~\ref{thm:asymptotics} from the introduction.

\begin{theorem}
For each arc $j$ of the signed WKB triangulation of~$\phi$, we have 
\[
\mathcal{Y}_j(\hbar)\cdot\exp(Z_{\gamma_j}/\hbar)\rightarrow1
\]
as $\hbar\rightarrow0$, $\Re(\hbar)>0$.
\end{theorem}

\begin{proof}
By Theorem~\ref{thm:main}, the restriction of $\mathcal{Y}_j(\hbar)$ to the domain $\mathbb{H}(\varepsilon)$ equals the Borel sum of the Voros symbol associated to the arc $j$. From the recursion relations used to define the 1-form $S_{\text{odd}}(z,\hbar)dz$, one sees that this expression can be written 
\[
S_{\text{odd}}(z,\hbar)dz=\frac{1}{\hbar}\sqrt{Q_0(z)}dz+\sum_{k\geq1}\hbar^kS_{\text{odd},k}(z)dz
\]
where the sum in the second term is a power series in~$\hbar$ without a constant term. Thus the Voros symbol can be written 
\[
\exp(V_{\gamma_j})=\exp\left(\frac{1}{\hbar}\int_{\gamma_j}\sqrt{Q_0(z)}dz\right)\exp\left(\sum_{k\geq1}\hbar^{k}\int_{\gamma_j}S_{\text{odd},k}(z)dz\right).
\]
The result follows.
\end{proof}

\section*{Acknowledgments}
\addcontentsline{toc}{section}{Acknowledgements}

The results of this paper are part of joint work with Tom~Bridgeland, who contributed many important ideas. I am grateful to Kohei~Iwaki, Tatsuya~Koike, Andrew~Neitzke, and Yoshitsugu~Takei for helpful discussions about exact WKB analysis.


\begin{thebibliography}{99}
\bibitem{KoikeSchafke} Koike, T. and Sch\"afke, R. On the Borel summability of WKB solutions of Schr\"odinger equations with polynomial potentials and its application. To appear.
\end{thebibliography}
\end{document}